\documentclass{article}
\usepackage{graphicx} 
\usepackage{amsmath}
\usepackage{amsfonts}
\usepackage{amsthm}
\usepackage{mathtools}

\usepackage[margin=30mm]{geometry}

\usepackage{float}

\newtheorem{lemma}{Lemma}
\newtheorem{theorem}{Theorem}
\newtheorem{definition}{Definition}
\newtheorem{proposition}{Proposition}
\newtheorem{corollary}{Corollary}

\newtheorem*{remark}{Remark}

\newcommand{\Lim}[1]{\raisebox{0.5ex}{\scalebox{0.8}{$\displaystyle \lim_{#1}\;$}}}
\newcommand{\bw}{\boldsymbol{\omega}}
\newcommand{\bgamma}{\boldsymbol{\gamma}}
\newcommand{\bx}{\boldsymbol{x}}
\newcommand{\by}{\boldsymbol{y}}

\title{Solving elliptic PDEs in unbounded domains}
\author{Doghonay Arjmand$^1$, Filip Marttala$^2$ \\
        \small $^{1}$Division of scientific computing, Uppsala University (doghonay.arjmand@it.uu.se)\\
        \small $^{2}$Division of scientific computing, Uppsala University (filip.marttala@it.uu.se)}
\date{}

\providecommand{\keywords}[1]
{
  \small	
  \textbf{\textit{Keywords---}} #1
}

\begin{document}

\maketitle

\begin{abstract}
    
An accurate approximation of solutions to elliptic problems in infinite domains is challenging from a computational point of view. This is due to the need to replace the infinite domain with a sufficiently large and bounded computational domain, and posing artificial boundary conditions on the boundary of the truncated computational geometry, which will then pollute the solution in an interior region of interest. For elliptic problems with periodically varying coefficients (with a possibly unknown period), a modelling strategy based on exponentially regularized elliptic problem was previously developed and analysed. The main idea was to replace the infinite domain periodic problem with a regularized elliptic problem posed over a finite domain, while retaining an accuracy decaying exponentially with respect to the size of the truncated domain. In this article, we extend the analysis to problems, where no structural assumptions on the coefficient are assumed. Moreover, the analysis here uncovers an interesting property of the right hand side in the Fourier domain for the method to converge fast for problems beyond periodicity.
\end{abstract}

\section{Introduction} \label{section:Introduction}



Partial differential equations (PDEs) posed over unbounded domains are relevant for many applications in science and engineering, including multiscale mechanics \cite{cioranescu1999introduction,pavliotis2008multiscale,weinan2007heterogeneous}, micro-magnetism \cite{aharoni2000introduction,arjmand2024hybrid,arjmand2020modelling}, wave propagation \cite{Duru_Kreiss_2012,PARDO2021219}, and plasma physics simulations \cite{VEDENYAPIN201119,DAO2024113009}. A common denominator in these applications is that although the main problem is posed over an infinite domain (or possibly a very large domain), we are often interested in obtaining the solution around a localized geometry; e.g., close to a local source. Since it is not feasible to numerically solve PDEs over unbounded domains, the domain will in most cases have to be truncated. Such a truncation, however, necessitates artificial boundary conditions on the boundary of the truncated domain, which may then deteriorate the accuracy of the solution inside a domain of interest. Starting from the celebrated result of Engquist and Majda, \cite{Engquist_Majda_77}, on absorbing boundary conditions, there has been a substantial progress in developing numerical methods for wave propagation problems in infinite domain, see e.g., \cite{berenger1994perfectly,Duru_Kreiss_2012,PARDO2021219} for techniques based on perfectly matched layers (PML), and \cite{nabizadeh2021kelvin,taflove2005computational} and the references therein for techniques based on coordinate transform. For elliptic problems, however, these issues have been partially adressed; under limited assumptions on the coefficients \cite{nabizadeh2021kelvin,gloria2011reduction,arjmand2016time,abdulle2023elliptic}.

In this article, we adopt a general mathematical standpoint and consider the following elliptic equation posed over $\mathbb{R}^d$
\begin{equation} \label{eqn_Main_Problem}\begin{split}
    -\nabla \cdot \left(  a(\boldsymbol{x}) \nabla u(\boldsymbol{x}) \right) &= g(\boldsymbol{x}), \quad \boldsymbol{x} \in \mathbb{R}^{d},
\end{split}\end{equation}
where $a(\boldsymbol{x})$ is a positive matrix function $a(\boldsymbol{x}):\mathbb{R}^d\rightarrow\mathbb{R}^{d\times d}$ bounded above and below and $g$ is assumed to have compact support. \footnote{The assumption of compact support for $g$ is needed only for a part of the theory, where we prove estimates for a standard approximation of \eqref{eqn_Main_Problem}. The overall analysis of the proposed approach, however, works under much milder decaying conditions on $g$, which is detailed in the analysis section.} Local $L^p$ estimates for \eqref{eqn_Main_Problem} can be found e.g., in \cite{Armstrong2019QuantitativeSH} (proposition $7.3$) under uniform continuity assumption on the coefficient matrix $a$; see also \cite{KANG20102643,Hofman_Kim_2007,gilbarg1977elliptic} for other relevant results about theory of elliptic regularity in $\mathbb{R}^d$. Such a model problem is relevant, for example, in homogenization theory \cite{cioranescu1999introduction}, where the coefficient $a$ is a rapidly varying function, and micro-magnetism where $a=1$, and $g=\nabla\cdot \boldsymbol{M}$ for a given magnetization function $\boldsymbol{M}:\Omega\rightarrow\mathbb{R}^d$, where $\Omega$ is a bounded subset of $\mathbb{R}^3$ filled with a ferromagnetic material, \cite{aharoni2000introduction}. Note that in the former, the coefficient $a$ is a high frequency function, whereas in the latter the coefficient is constant (slowly varying). Motivated by these applications, the analysis in this paper covers theoretical results both for oscillatory as well as slowly varying coefficients. In particular, for high-frequency coefficients, we establish uniform upper bounds with respect to the wavelength of the oscillations in the coefficient $a$. The method described in this paper may be applicable also for other areas than homogenization and magnetization problems, and our analysis assumes a right hand side $g$ with certain properties which will be made clear in the analysis. 

In order to make the infinite domain problem \eqref{eqn_Main_Problem} amenable to a numerical treatment, one typically needs to truncate the infinite domain $\mathbb{R}^d$ into a finite domain as follows:
\begin{equation}\label{eq:main_problem_truncated}
    -\nabla \cdot \left(  a(\boldsymbol{x}) \nabla u_{R}(\boldsymbol{x}) \right) = g(\boldsymbol{x}), \quad \boldsymbol{x} \in K_R:=(-R/2,R/2)^d,
\end{equation}
where \eqref{eq:main_problem_truncated} is equipped with some suitable boundary conditions. As will be shown in Section \ref{section:regtermmotivation}, equation \eqref{eq:main_problem_truncated} is not a very efficient way of approximating \eqref{eqn_Main_Problem}, and better strategies are needed to approximate the solution $u$ to the problem \eqref{eqn_Main_Problem}. 

In the context of periodic homogenization, promising methodologies have been proposed over the last two decades to efficiently approximate the solution $u$, by solving variants of the finite domain problem \eqref{eq:main_problem_truncated}, and establishing error estimates for the difference \footnote{In applications related with homogenization one typically requires $H^1$ estimates, but in this paper we will focus only on $L^2$ estimates} $\| u - u_{R} \|$  over a subset $K_L$ of the truncated domain $K_R$. In particular, the desired goal has been to develop and analyse models whose error scale as $\frac{1}{R^q}$ for a large $q$. In \cite{yue2007local}, Yue and E demonstrate that choosing periodic boundary conditions (instead of homogeneous Dirichlet or Neuman conditions) for the truncated problem \eqref{eq:main_problem_truncated} can give an improvement in the prefactor, but does not improve the first order convergence rate in $\frac{1}{R}$. Another set of methods, leading to higher order rates, rely on modifying the truncated problem \eqref{eq:main_problem_truncated}. In \cite{blanc2009improving}, Blanc and Le Bris use an integral constraint on the gradient of the solution, resulting in second order convergence. Fourth order convergence is attained in \cite{gloria2011reduction} where Gloria modifies the truncated problem \eqref{eq:main_problem_truncated} by adding a zeroth order term. The idea behind this added term is to improve the decay of the Green's function, resulting in a reduced effect of inaccurate boundary conditions on the interior solution over $K_L \subset K_R$. While Gloria's method is fourth order and comes with almost no added computational cost, the fourth order convergence is in general only observed for large domains, and the observed pre-asymptotic rate is less than fourth order. A somewhat different approach is found in \cite{arjmand2016time} where instead of solving \eqref{eq:main_problem_truncated} directly, Arjmand and Runborg instead solve the wave equation over $K_R$, and compute a temporal average to approximate the solution $u$ up to $O(\frac{1}{R^q})$ accuracies for arbitrarily large values for $q$. The idea behind this method is that since the main error of \eqref{eq:main_problem_truncated} comes from the boundaries, for the wave equation, these errors will propagate inwards towards the rest of the domain at finite speed. Then for a sufficiently large domain $K_R$, there is a subdomain $K_L \subset K_R$ that is unaffected by the boundary errors. While this approach allows for arbitrarily high convergence rates, due to completely annihilating the boundary errors, it requires solving a time-dependent problem which can be expensive. Further, if the maximum wave speed, given by $\sqrt{\lVert a\rVert_{L^\infty(\mathbb{R}^d)}}$, is large, then the computational geometry needs to be taken very large too, which is prohibitive from a computational point of view. Finally, central to the goal of the present article is the idea of regularizing the finite domain problem \eqref{eq:main_problem_truncated} by exponential power of an elliptic operator, \cite{abdulle2023elliptic}, such that the added term results in a Green's function with a Gaussian decay in space. Upon choosing the parameters optimally, this method allows for exponential convergence with respect to $R$ (in the periodic setting), at the expense of computing a one-dimensional dense matrix exponential (independent of the dimension of the original problem); see also \cite{Carney_etal_2024,abdulle2021parabolic} for other relevant approaches well-suited for applications in periodic homogenization.\\



Considering the advantages that the exponential regularization approach \cite{AA_DA_EP_2019__357_6_545_0,abdulle2023elliptic}, proposed by Abdulle et.al., has in periodic homogenization, the main goal of this paper is to establish local $L^2$ error estimates for the difference between the solution of the exponential regularization approach and the solution $u$ of the infinite domain problem \eqref{eqn_Main_Problem}, while relaxing the periodicity assumption on the coefficient $a$. This paper is structured as follows: In Section \ref{section:regtermmotivation}, we introduce a naive way of approximating the problem \eqref{eqn_Main_Problem}, and provide a  quantitative analysis of the convergence rate. In Section \ref{section:decayimprov}, we introduce the exponential regularization approach and prove that the solution to the model problem convergences much faster than the naive approach. Finally, in Section \ref{sec:NumericalResults}, we provide numerical evidence in two and three dimensions corroborating our theoretical findings. 

\subsection{Preliminaries} 
Throughout this paper we will be using the following definitions, notations, and conventions:
\begin{itemize}
    \item We use the notation $K_R$ to denote a cube in $\mathbb{R}^d$ given by $\left(-\frac{R}{2},\frac{R}{2}\right)^d$, and $K=K_1$ denotes the unit cube $\left(-\frac{1}{2},\frac{1}{2}\right)^d$.
    \item $L^p(\Omega)$ denotes the standard $L^p$ spaces whose norm is given by  
    \begin{equation*}
    L^p(\Omega)=\{f :\int_\Omega \lvert f(\boldsymbol{x})\rvert^p \; d\boldsymbol{x} < \infty\}. 
    \end{equation*}
    Moreover, $L^p(\Omega,w(\boldsymbol{x}))$ denotes the weighted $L^p$ space with the norm
    \begin{equation*}
    L^p(\Omega,w(\boldsymbol{x}))=\{f :\int_\Omega \lvert f(\boldsymbol{x})\rvert^p w(\boldsymbol{x}) \;d\boldsymbol{x} < \infty\}.
    \end{equation*}
    \item The Bohner space $L^p(0,T;X)$, where $X$ is a Banach space is associated with the norm
    \begin{equation*}
        \| f \|_{L^p(0,T;X)} := \left( \int_{0}^{T} \| f \|^{p}_{X} \; dt \right)^{\frac{1}{p}}.
    \end{equation*}
    \item By $C^{k,r}$ we denote denote the space of $k$-times continuously differentiable Hölder continuous functions with exponent $0< r \leq 1$, which consists of functions $f$ such that 
\begin{equation*}
    \| f \|_{C^{k,r}} := \| f \|_{C^k}  + \max_{|\bgamma|=k} |D^{\bgamma} f|_{C^{0,r}},
\end{equation*}
and 
\begin{equation*}
    |f|_{C^{0,r}}:= \sup_{\bx\neq \by} \dfrac{|f(\bx) - f(\by)|}{\| \bx - \by \|^{r}}.
\end{equation*}
    \item The letter $C$ refers to a constant independent of the computational parameters $R$,$L$, and $T$. The value of $C$ may change between each step of a proof.
    \item Boldface letters implies the variable is vector-valued. Plain letters implies the variable is scalar-valued.
    \item For a PDE of the form $Lu=f$, the function $G(\boldsymbol{x};\boldsymbol{y})$ satisfying $LG(\boldsymbol{x};\boldsymbol{y}) = \delta(\boldsymbol{y}-\boldsymbol{x})$, where $\delta(\boldsymbol{x})$ is the dirac delta distribution, is called the Green's function. In particular, we will refer to $G_{\infty}$ and $G_R$ for the Green's function to \eqref{eqn_Main_Problem} and \eqref{eq:main_problem_truncated} respectively. For time dependent problems the Green's function is denoted by $G(t,\boldsymbol{x};\boldsymbol{y})$ 
    \item We say that $a\in\mathcal{M}(\alpha,\beta,\Omega)$ if $a_{ij}=a_{ji}$, $a\in[L^\infty(\Omega)]^{d\times d}$ and there are constants $0<\alpha\leq\beta$ such that

    \begin{equation}
        \alpha\lvert \boldsymbol{\zeta}\rvert^2\leq\boldsymbol{\zeta}\cdot a(\boldsymbol{x})\boldsymbol{\zeta} \leq \beta \lvert \boldsymbol{\zeta}\rvert^2, \quad \text{a.e. for } \boldsymbol{x}\in\Omega,\forall\boldsymbol{\zeta}\in\mathbb{R}^d.
    \end{equation}
    \item We denote the Fourier transform of a function $f$ by $\hat{f}$.

We finalize this section by providing two known Theorems, which will be used in the analysis section.

\begin{theorem}
    
Given a function $g(\boldsymbol{x}) \in L^1(1+ |x|^{k+1})$ we can decompose $g$ as 
\begin{align*}
    g(\boldsymbol{x}) = \sum_{|\bgamma| \leq k} \dfrac{(-1)^{|\bgamma|}}{|\bgamma|!} \int_{\mathbb{R}^d} g(\boldsymbol{x}) \boldsymbol{x}^{\bgamma} \; d\boldsymbol{x} D^{\bgamma} \delta_0  + \sum_{|\bgamma|  = k+1} D^{\bgamma}F_{\bgamma}, 
\end{align*}

where $\| F_{\bgamma} \|_{L^1(\mathbb{R}^d)} \leq C_d \| (1+|\boldsymbol{x}|^{k+1}) g    \|_{L^{1}(\mathbb{R}^d)}$, and $\delta_0$ is the Dirac distribution centered at the origin.
\label{theorem:gdecomp}
\end{theorem}

\begin{proof}
     see \cite{DUOANDIKOETXEA1992}.
\end{proof}

\begin{theorem}[Theorem 1.2 in \cite{F_O_Porper_1984}] \label{Eidelman_Survey}
    Let $G$ be the Green's function corresponding to the equation \eqref{eqn_InfDomain}, and assume that the coefficient $a \in C^{k,r}(\mathbb{R}^d)$ is positive and bounded. Then 
    \begin{align*}
        |\partial_{\boldsymbol{x}}^{\bgamma} G(t,\boldsymbol{x};\boldsymbol{y})| \leq C t^{-(d + k+ 1)/2} e^{-c \frac{|\boldsymbol{x}-\boldsymbol{y}|^2}{t}}, \quad t>0, |\bgamma| \leq k,
    \end{align*}
where the constants C and c do not depend on $t,x,y$ but may depend on $d,a$, and $r$. 
\label{theorem:Gpartderivbound}
\end{theorem}

\end{itemize}

\section{Convergence of elliptic equations with no regularization} \label{section:regtermmotivation}

In this section, we consider a naive truncation of the elliptic PDE (1) to a bounded domain $K_R$ of size $\mathcal{O}(R^d)$, and study the error analysis for $d \geq 3$. As we will see, the analysis shows a poor decay of the error with respect to increasing $R$. This implies that very large computational domains are necessary to reach practical tolerances for the error. We begin by considering the problem 
\begin{equation}\begin{split}
    -\nabla\cdot(a(\boldsymbol{x})\nabla \Tilde{u}_R(\boldsymbol{x})&=g(\boldsymbol{x}), \text{    }\boldsymbol{x}\in K_R,\\
    \Tilde{u}_R(\boldsymbol{x}) &=0, \text{    }\boldsymbol{x}\text{ on }\partial K_R.
    \label{eq:finitedomainproblem}
\end{split}\end{equation}
Recalling the definition $K_R=(-\frac{R}{2},\frac{R}{2})^d\subset \mathbb{R}^d$, we define $E_{R}:=K_R\setminus K_{\frac{3}{4}R}$. Moreover, let $L<\frac{R}{4}$ such that $K_L\subset K_{\frac{1}{4}R}$. We then have the following estimate:

\begin{theorem} \label{theorem:mainthrmnaive}
    Let $u$ solve \eqref{eqn_Main_Problem} and $\Tilde{u}_R$ solve \eqref{eq:finitedomainproblem} with coefficients $a(\boldsymbol{x})\in\mathcal{M}(\alpha,\beta,\mathbb{R}^d)$ and $d\geq 3$. Let $L< \frac{R}{4}$ such that $K_L\subset K_{\frac{R}{4}}$. Further, let $\text{supp}(g(\boldsymbol{x}))=\Omega\subset K_L$ and $g\in L^\infty(\Omega)$. Then

    \begin{equation}
        \lVert u-\Tilde{u}_R\rVert_{L^\infty(K_L)}\leq C\frac{1}{R^{d-2}}\lVert g\rVert_{L^\infty(\Omega)},
    \end{equation}
    where $C$ is a constant dependent only on $\beta$,$d$, and $\Omega$.
\end{theorem}

We note that for dimension $d=3$ we only have first order decay with increasing domain size. Considering the large cost associated with increasing the domain size in 3 dimensions this is a problematically slow decay. While higher dimensions allow for faster decay of the boundary error, it also means a higher cost incurred from increasing the domain size. Further, since the case $d=3$ is very common for physical reasons, we will in many applications end up with a first order decay.\\

Before we can provide a proof for theorem \ref{theorem:mainthrmnaive} some lemmas need to be stated.

\begin{lemma}
    Let $G_\infty(\boldsymbol{x};\boldsymbol{y})$ and $G_R(\boldsymbol{x};\boldsymbol{y})$ denote the Green's functions of \eqref{eqn_Main_Problem} and \eqref{eq:finitedomainproblem} with coefficients $a(\boldsymbol{x})\in\mathcal{M}(\alpha,\beta,\mathbb{R}^d$) and $d\geq 3$, respectively. Then the estimate 

    \begin{align}
        0\leq G_R(\boldsymbol{x};\boldsymbol{y})\leq G_\infty(\boldsymbol{x};\boldsymbol{y})&=C_\beta\frac{1}{\lvert \boldsymbol{x}-\boldsymbol{y}\rvert^{d-2}}
    \end{align}
    holds.
    \label{lemma:nasharonsonelliptic}
\end{lemma}

\begin{proof}
    Consider the equations defining $G_R$ and $G_\infty$,

    \begin{equation}\begin{split}
        -\nabla\cdot(a(\boldsymbol{x})\nabla G_R(\boldsymbol{x};\boldsymbol{y})&=\delta(\boldsymbol{y}-\boldsymbol{x}),\text{ in }K_R,\\
        G_R(\boldsymbol{x};\boldsymbol{y})&=0,\text{ on }\partial K_R
    \end{split}\end{equation}
    and
    \begin{equation}
       - \nabla\cdot(a(\boldsymbol{x})\nabla G_\infty(\boldsymbol{x};\boldsymbol{y}))=\delta(\boldsymbol{y}-\boldsymbol{x}),\text{ in }\mathbb{R}^d.
    \end{equation}
    We now define a new function $H=G_\infty-G_R$. By linearity, $H$ then satisfies
    \begin{equation}\begin{split}
       - \nabla\cdot(a(\boldsymbol{x})\nabla H(\boldsymbol{x};\boldsymbol{y}))&=0,\text{ in }K_R,\\
        H(\boldsymbol{x};\boldsymbol{y})&=G_\infty(\boldsymbol{x};\boldsymbol{y}),\text{ on }\partial K_R.
    \end{split}\end{equation}
    By positivity of $G_{\infty}$ and the maximum principle it follows that $G_\infty\geq G_R$.
\end{proof}

\begin{remark}
    The above estimate itself does not require $d\geq 3$ but the explicit upper bound in Lemma \eqref{lemma:nasharonsonelliptic} does since the Green's function of \eqref{eqn_Main_Problem} is of another form for $d\leq 2$.
\end{remark}

\begin{corollary}
    Let $G_\infty(\boldsymbol{x};\boldsymbol{y})$ and $G_R(\boldsymbol{x};\boldsymbol{y})$ denote the Green's functions of \eqref{eqn_Main_Problem} and \eqref{eq:finitedomainproblem} with coefficients $a(\boldsymbol{x})\in\mathcal{M}(\alpha,\beta,\mathbb{R}^d$) and $d\geq 3$, respectively. Then, for all $\boldsymbol{x}\in K_L$, $\boldsymbol{y}\in E_{R}$,

    \begin{align}
        \lvert G_R(\boldsymbol{x};\boldsymbol{y}) \rvert&\leq C_\beta\frac{1}{\left\lvert\frac{3}{8}R-\frac{L}{2}\right\rvert^{d-2}}.
        \label{eq:linfboundG1}
    \end{align}
    \label{lemma:ellipticboundG}
\end{corollary}

\begin{proof}
    This is a direct consequence of Lemma \ref{lemma:nasharonsonelliptic}, and the definitions of $K_L$ and $E_R$.


\end{proof}

\begin{lemma}
    Using the assumptions from Lemma \ref{lemma:ellipticboundG}, the following bounds on the $L^2$ norms of the Green's functions hold.

    \begin{align}
        \lVert G_R(\boldsymbol{x};\cdot) \rVert_{L^2(E_{R})}&\leq C_\beta\lvert E_{R}\rvert^\frac{1}{2}\frac{1}{\left\lvert\frac{3}{8}R-\frac{L}{2}\right\rvert^{d-2}}, \label{eq:l2boundG1}\\
        \lVert \nabla G_R(\boldsymbol{x};\cdot) \rVert_{L^2(E)}&\leq \frac{C_\beta}{R}\lvert E_{R}\rvert^\frac{1}{2}\frac{1}{\left\lvert\frac{3}{8}R-\frac{L}{2}\right\rvert^{d-2}}.\label{eq:l2boundG2}
    \end{align}
    \label{lemma:l2boundG}
    
\end{lemma}

\begin{proof}
    The first inequality follows directly from \eqref{eq:linfboundG1}. The second we can study by looking at the PDE for the Green's function

    \begin{equation}\begin{split}
       - \nabla\cdot(a(\boldsymbol{x})\nabla G_R(\boldsymbol{x};\boldsymbol{y}))&=\delta(\boldsymbol{y}-\boldsymbol{x}),\text{    in }K_R,\\
        \text{  }G_R(\boldsymbol{x;\boldsymbol{y}})&=0\text{ on }\partial K_R.
        \label{eq:fundsolutionellipticpde}
    \end{split}
    \end{equation}
Furthermore, we define a smooth function $\eta(\boldsymbol{x}):K_R\rightarrow [0,1]$ such that

    \begin{equation}
        \eta(\boldsymbol{x}) = \begin{cases}
                1\text{ in }E_{R}=K_R\setminus K_{\frac{3}{4}R},\\
            0\text{ in }K_{\frac{R}{4}},
        \end{cases}
    \end{equation}
and $\lvert\nabla \eta(\boldsymbol{x})\rvert\leq\frac{C}{R}$. It is important to note that for all $\boldsymbol{x}\in K_L$, we have $\eta(\boldsymbol{x})=0$.
    We may now test \eqref{eq:fundsolutionellipticpde} against $\eta(\boldsymbol{x})^2G_R(\boldsymbol{x};\boldsymbol{y})$. We get

    \begin{equation}
       - \int_{K_R}\nabla\cdot(a(\boldsymbol{y})\nabla G_R(\boldsymbol{x};\boldsymbol{y}))\eta(\boldsymbol{y})^2G_{R}(\boldsymbol{x};\boldsymbol{y})d\boldsymbol{y}=\int_{K_R}\delta(\boldsymbol{y}-\boldsymbol{x})\eta(\boldsymbol{y})^2G_{R}(\boldsymbol{x};\boldsymbol{y})d\boldsymbol{y}.
    \end{equation}
    The right hand side vanishes since $\boldsymbol{x}\in K_L$. Using integration by parts and the Dirichlet boundary conditions then gives
    \begin{equation}
    \begin{split}
        0=&\int_{K_R}a(\boldsymbol{y})\nabla G_R(\boldsymbol{x};\boldsymbol{y})\cdot\nabla(\eta(\boldsymbol{y})^2G_R(\boldsymbol{x};\boldsymbol{y}))d\boldsymbol{y}=\\=&\int_{K_R}a(\boldsymbol{y})\nabla(\eta(\boldsymbol{y}) G_R(\boldsymbol{x};\boldsymbol{y}))\cdot\nabla(\eta(\boldsymbol{y}) G_R(\boldsymbol{x};\boldsymbol{y}))d\boldsymbol{y}+\\&-\int_{K_R}a(\boldsymbol{y})G_R(\boldsymbol{x};\boldsymbol{y})^2\nabla\eta(\boldsymbol{y})\cdot\nabla\eta(\boldsymbol{y})d\boldsymbol{y}.
    \end{split}
    \end{equation}
    Therefore,
    \begin{equation}\begin{split}
        &\int_{K_R}a(\boldsymbol{y})\nabla(\eta(\boldsymbol{y}) G_R(\boldsymbol{x};\boldsymbol{y}))\cdot\nabla(\eta(\boldsymbol{y}) G_R(\boldsymbol{x};\boldsymbol{y}))d\boldsymbol{y}=\\&=\int_{K_R}a(\boldsymbol{y})G_R(\boldsymbol{x};\boldsymbol{y})^2\nabla\eta(\boldsymbol{y})\cdot\nabla\eta(\boldsymbol{y})d\boldsymbol{y}.
    \end{split}\end{equation}
    Since $\alpha\leq a\leq \beta$ this implies
    \begin{equation}
        \int_{K_R}\nabla(\eta(\boldsymbol{y}) G_R(\boldsymbol{x};\boldsymbol{y}))\cdot\nabla(\eta(\boldsymbol{y}) G_R(\boldsymbol{x};\boldsymbol{y}))d\boldsymbol{y}\leq\int_{K_R}\frac{\beta}{\alpha}G_R(\boldsymbol{x};\boldsymbol{y})^2\nabla\eta(\boldsymbol{y})\cdot\nabla\eta(\boldsymbol{y})d\boldsymbol{y}.
    \end{equation}
    Using the inequality \eqref{eq:linfboundG1} and $\lvert\nabla\eta(\boldsymbol{x})\rvert\leq\frac{C}{R}$, we have

    \begin{equation}
    \begin{split}
        \int_{E_{R}}\lvert\nabla G_R(\boldsymbol{x};\boldsymbol{y})\rvert^2d\boldsymbol{y}&=\int_{E_{R}}\nabla(\eta(\boldsymbol{y})G_R(\boldsymbol{x};\boldsymbol{y}))\cdot\nabla(\eta(\boldsymbol{y})G_R(\boldsymbol{x};\boldsymbol{y})) d\boldsymbol{y}\leq\\
        &\leq\int_{E_{R}}\frac{\beta}{\alpha}G_R(\boldsymbol{x};\boldsymbol{y})^2\nabla\eta(\boldsymbol{y})\cdot\nabla\eta(\boldsymbol{y})d\boldsymbol{y}\leq\\
        &\leq \frac{C}{R^2}\lvert E_{R} \rvert\frac{1}{\left\lvert\frac{3}{8}R-\frac{L}{2}\right\rvert^{2d-4}}.
    \end{split}
    \end{equation}
    Taking the square root of both sides then ends the proof. 
\end{proof}
We may now state some bounds on $u$.
\begin{lemma}
    Let $u$ solve \eqref{eqn_Main_Problem} for coefficients $a(\boldsymbol{x})\in\mathcal{M}(\alpha,\beta,\mathbb{R}^d)$ and $d\geq 3$ and $\text{supp}(g(\boldsymbol{x}))=\Omega\subset K_L$, $g\in L^\infty(\Omega)$. Then, for all $\boldsymbol{x}\in E_{R}$,
    \begin{align}
        \lvert u(\boldsymbol{x})\rvert &\leq C_\beta\lvert\Omega\rvert\frac{1}{\left\lvert\frac{3}{8}R-\frac{L}{2}\right\rvert^{d-2}}\lVert g\rVert_{L^\infty(\Omega)},\\
        \lvert \nabla u(\boldsymbol{x})\rvert &\leq C_{\beta,d}\lvert\Omega\rvert\frac{1}{\left\lvert\frac{3}{8}R-\frac{L}{2}\right\rvert^{d-1}}\lVert g\rVert_{L^\infty(\Omega)},
    \end{align}
    where $C_{\beta}$ is the constant appearing also in Corollary \ref{lemma:ellipticboundG}, and $C_{\beta,d}$ depends only on $\beta$ and the dimension $d$.
\end{lemma}

\begin{proof}
    Let $G_\infty(\boldsymbol{x};\boldsymbol{y})$ denote the Green's function of \eqref{eqn_Main_Problem}. Then by definition of the Green's function and using $\text{supp}(g(\boldsymbol{x}))=\Omega$, we have

    \begin{equation}
        u(\boldsymbol{x}) =\int_{\mathbb{R}^d}G_\infty(\boldsymbol{x};\boldsymbol{y})g(\boldsymbol{y})d\boldsymbol{y}=\int_\Omega G_\infty(\boldsymbol{x};\boldsymbol{y})g(\boldsymbol{y})d\boldsymbol{y}.
    \end{equation}
    By noting that the integral is only over $\Omega\subset K_L$ and restricting the values of $\boldsymbol{x}$ to $\boldsymbol{x}\in E_{R}$, together with using the symmetry of $G$ allows us to use Corollary \ref{lemma:ellipticboundG} to get
    \begin{equation}
        \lvert u(\boldsymbol{x})\rvert \leq C_\beta\lVert g\rVert_{L^\infty(\Omega)}\int_\Omega\frac{1}{\left\lvert\frac{3}{8}R-\frac{L}{2}\right\rvert^{d-2}}d\boldsymbol{y}=C_\beta\lvert\Omega\rvert\frac{1}{\left\lvert\frac{3}{8}R-\frac{L}{2}\right\rvert^{d-2}}\lVert g\rVert_{L^\infty(\Omega)}.
    \end{equation}
    The bound for $\lvert\nabla u\rvert$ can be proven in the same way except for starting with
    \begin{equation}
        \lvert \nabla u(\boldsymbol{x})\rvert =\int_{\mathbb{R}^d}\nabla G_\infty(\boldsymbol{x};\boldsymbol{y})g(\boldsymbol{y})d\boldsymbol{y}=\int_\Omega \nabla G_\infty(\boldsymbol{x};\boldsymbol{y})g(\boldsymbol{y})d\boldsymbol{y}.
    \end{equation}
    Noting that $\nabla G_\infty(\boldsymbol{x};\boldsymbol{y})=C_\beta\frac{1}{\lvert \boldsymbol{x}-\boldsymbol{y}\rvert^{d-1}}$
\end{proof}

\begin{corollary}
    These bounds can also be used to bound the $L^2$-norm of $u$ and $\nabla u$, giving us

    \begin{align}
        \lVert u\rVert_{L^2(E_{R})} &\leq C_\beta\lvert\Omega\rvert\lvert E_{R}\rvert^{\frac{1}{2}}\frac{1}{\left\lvert\frac{3}{8}R-\frac{L}{2}\right\rvert^{d-2}}\lVert g\rVert_{L^\infty(\Omega)},\\
        \lVert \nabla u\rVert_{L^2(E_{R})} &\leq C_{\beta,d}\lvert\Omega\rvert\lvert E_{R}\rvert^\frac{1}{2}\frac{1}{\left\lvert\frac{3}{8}R-\frac{L}{2}\right\rvert^{d-1}}\lVert g\rVert_{L^\infty(\Omega)}.
    \end{align}
    \label{corollary:l2boundu}
\end{corollary}

\begin{remark}
    For these bounds it is useful to note that $\lvert E_{R}\rvert =\frac{4^d-3^d}{4^d}R^{d}$ since by definition, $\lvert E_{R} \rvert= \lvert K_{R}\rvert - \lvert K_{\frac{3}{4}R}\rvert$.    
    \label{remark:Emsize}
\end{remark}
We may now prove theorem \ref{theorem:mainthrmnaive}.
\begin{proof}[Proof of Theorem \ref{theorem:mainthrmnaive}]
We start by defining 
\begin{equation}
    e_R(\boldsymbol{x})=\rho(\boldsymbol{x}) u(\boldsymbol{x})-\Tilde{u}_R(\boldsymbol{x})    
\end{equation}
where $\rho\in C_c^1(K_R)$ is defined as
\begin{equation}
    \rho(\boldsymbol{x}) = \begin{cases}
        1,\text{    }\boldsymbol{x}\in K_{\frac{3}{4}R},\\
        0,\text{    }\boldsymbol{x}\text{ on } \partial K_R,      
    \end{cases}
\end{equation}
and $\lvert\nabla \rho(\boldsymbol{x})\rvert \leq\frac{C}{R}$ on $E$. Then $e_R$ satisfies
\begin{equation}\begin{split}
    -\nabla\cdot(a(\boldsymbol{x})\nabla e_R(\boldsymbol{x}))&=F(\boldsymbol{x}), \text{    }\boldsymbol{x}\in K_R,\\
    e_R(\boldsymbol{x})&=0, \text{    }\boldsymbol{x}\text{ on } \partial K_R,
\end{split}\end{equation}
where 
\begin{equation}
    F(\boldsymbol{x}) = (\rho(\boldsymbol{x})-1)g(\boldsymbol{x})-a(\boldsymbol{x})\nabla\rho(\boldsymbol{x})\cdot\nabla u(\boldsymbol{x})-\nabla\cdot(a(\boldsymbol{x})u(\boldsymbol{x})\nabla\rho(\boldsymbol{x})).
\end{equation}
Clearly $e_R(\boldsymbol{x}) = u(\boldsymbol{x})-\Tilde{u}_R(\boldsymbol{x})$ in $K_L\subset K_{\frac{3}{4}R}$, and therefore, $\lVert e_R\rVert_{L^\infty(K_L)}=\lVert u-\Tilde{u}_R\rVert_{L^\infty(K_L)}$. Denote by $G_R(\boldsymbol{x};\boldsymbol{y})$ the Green's function to \eqref{eq:finitedomainproblem} and let $G_\infty(\boldsymbol{x};\boldsymbol{y})$ denote the Green's function to \eqref{eqn_Main_Problem}. Then by the definition of the Green's function we have
\begin{equation}\begin{split}
    e_R(\boldsymbol{x}) =& \int_{K_R} G_R(\boldsymbol{x};\boldsymbol{y})F(\boldsymbol{y})d\boldsymbol{y}\\
    =& \int_{K_R} (\rho(\boldsymbol{y})-1)G_R(\boldsymbol{x};\boldsymbol{y})g(\boldsymbol{y})+\\&- \int_{K_R} G_R(\boldsymbol{x};\boldsymbol{y})a(\boldsymbol{y})\nabla\rho(\boldsymbol{y})\cdot\nabla u(\boldsymbol{y})d\boldsymbol{y}+\\&-\int_{K_R} G_R(\boldsymbol{x};\boldsymbol{y})\nabla\cdot(a(\boldsymbol{y})u(\boldsymbol{y})\nabla\rho(\boldsymbol{y}))d\boldsymbol{y}.
    \label{eq:er1}
\end{split}\end{equation}
We note that the first term vanishes because $g$ has support only on $K_L$ and $\rho(\boldsymbol{x})-1=0$ for all $\boldsymbol{x}\in K_{\frac{3}{4}R}\supset K_L$. We also note that because $\nabla\rho=0\text{  }\forall \boldsymbol{x}\notin E$, we can rewrite \eqref{eq:er1} as
\begin{equation}\begin{split}
    e_R(\boldsymbol{x})=& -\int_{E_{R}} G_R(\boldsymbol{x};\boldsymbol{y})a(\boldsymbol{y})\nabla\rho(\boldsymbol{y})\cdot\nabla u(\boldsymbol{y})d\boldsymbol{y}+\\&-\int_{E_{R}} G_R(\boldsymbol{x};\boldsymbol{y})\nabla\cdot(a(\boldsymbol{y})u(\boldsymbol{y})\nabla\rho(\boldsymbol{y}))d\boldsymbol{y}.
    \label{eq:er2}
\end{split}\end{equation}
Using $a(\boldsymbol{x})\leq \beta$, $\lvert\nabla\rho(\boldsymbol{x})\rvert\leq\frac{C}{R}$ and Cauchy-Schwartz we get
\begin{equation}
    \lvert e_R(\boldsymbol{x})\rvert\leq \frac{C\beta}{R}\left(\lVert G_R(\boldsymbol{x};\cdot)\rVert_{L^2(E_{R})}\lVert \nabla u\rVert_{L^2(E_{R})}+\lVert \nabla G_R(\boldsymbol{x},\cdot)\rVert_{L^2(E_{R})}\lVert u\rVert_{L^2(E_{R})}\right).
    \label{eq:er3(erabs)}
\end{equation}
By restricting $\boldsymbol{x}$ to $K_L$, and upon using Lemma \ref{lemma:l2boundG}, and Corollary \ref{corollary:l2boundu}, we obtain
\begin{equation}
    \lvert e_R(\boldsymbol{x})\rvert \leq \frac{C}{R}\left(\lvert E_{R}\rvert\frac{1}{\left\lvert\frac{3}{8}R-\frac{L}{2}\right\rvert^{2d-3}}\lVert g\rVert_{L^\infty(\Omega)}+\frac{\lvert E_{R}\rvert}{R}\frac{1}{\left\lvert\frac{3}{8}R-\frac{L}{2}\right\rvert^{2d-4}}\lVert g\rVert_{L^\infty(\Omega)}\right).
\end{equation}
Moreover, since $\lvert E_{R}\rvert=\frac{4^d-3^d}{4^d}R^d$ we get
\begin{equation}
    \lvert e_R(\boldsymbol{x})\rvert \leq C\left(R^{d-1}\frac{1}{\left\lvert\frac{3}{8}R-\frac{L}{2}\right\rvert^{2d-3}}\lVert g\rVert_{L^\infty(\Omega)}+R^{d-2}\frac{1}{\left\lvert\frac{3}{8}R-\frac{L}{2}\right\rvert^{2d-4}}\lVert g\rVert_{L^\infty(\Omega)}\right).
\end{equation}
Choosing $L\propto R$ then finally gives
\begin{equation}
    \lvert e_R(\boldsymbol{x})\rvert \leq C\frac{1}{R^{d-2}}\lVert g\rVert_{L^\infty(\Omega)}
\end{equation}
for all $\boldsymbol{x}\in K_L$.
\end{proof}
\begin{remark}
    Any choice of $L$ that does not cause $L$ to grow faster than $R$ (which would not be allowed, since at some point the condition $L<\frac{1}{4}R$ would  be violated) leads to the same rate of convergence.
\end{remark}

\section{Improving the Decay by Regularization} \label{section:decayimprov}
As we saw in the previous section, a simple truncation of the infinite domain $\mathbb{R}^d$ into $K_R$, and choosing homogeneous Dirichlet boundary conditions, results in a slow decay (of type $1/R^{d-2}$) for the $L^2$ norm of the error. To improve over this slow decay, we will use another approach based on exponential regularization of elliptic PDEs, first proposed in \cite{abdulle2023elliptic}, whose analysis has been addressed in the context of periodic homogenization. In the present work, we relax the periodicity assumption on the coefficients, and still prove rates decaying much faster than $1/R^{d-2}$. In \cite{abdulle2023elliptic}, the authors begin by introducing the parabolic equation
\begin{align} \label{eqn_InfDomain}
    \partial_t v(t,\boldsymbol{x}) -\nabla \cdot \left(  a(\boldsymbol{x}) \nabla v(t,\boldsymbol{x}) \right) &= 0, \quad (t,\boldsymbol{x}) \in (0,T] \times \mathbb{R}^{d}, \\
    v(0,\boldsymbol{x}) &= g(\boldsymbol{x}). \nonumber
\end{align}
In order to link this problem to \eqref{eqn_Main_Problem} we further introduce $v_T(\boldsymbol{x}) = \int_0^{T} v(t,\boldsymbol{x}) \; dt$. Notice then that if $v$ is decaying in time then $v_{\infty}(\boldsymbol{x})= u(\boldsymbol{x})$ where $u(\boldsymbol{x})$ solves \eqref{eqn_Main_Problem}. Next we introduce $u_T(\boldsymbol{x})  = \int_{0}^{T} w(t,\boldsymbol{x}) \; dt$, where $w$ solves
\begin{align}
    \partial_t w(t,\boldsymbol{x}) -\nabla \cdot \left(  a(\boldsymbol{x}) \nabla w(t,\boldsymbol{x}) \right) &= 0, \quad (t,\boldsymbol{x}) \in (0,T] \times K_{R},\label{eq:heatequationKR(wdef)} \\
    w(0,\boldsymbol{x}) &= g(\boldsymbol{x}) \nonumber\\
    w(t,\boldsymbol{x})  &= 0 \quad \text{ on } (0,T] \times \partial K_R.\nonumber
\end{align}
We observe that $u_T$ satisfies 
\begin{equation} \label{eqn_Modified_Problem}
    -\nabla \cdot \left(  a(\boldsymbol{x}) \nabla u_T(\boldsymbol{x}) \right) = g(\boldsymbol{x}) - [e^{-T\mathcal{A}} g](\boldsymbol{x}), \quad \boldsymbol{x} \in K_R,
\end{equation}
where $\mathcal{A}$ is the differential operator defined as $\mathcal{A}u=\nabla\cdot (a(\boldsymbol{x})\nabla u)$ and $[e^{-T\mathcal{A}} g](\boldsymbol{x})$ is the solution at time $T$ of the equation \eqref{eq:heatequationKR(wdef)}. We refer to the extra term in the right hand side of \eqref{eqn_Modified_Problem} the exponential regularization term. We now want to show that \eqref{eqn_Modified_Problem} approximates \eqref{eqn_Main_Problem} with a better accuracy than \eqref{eq:main_problem_truncated}. Analogous to the analysis in the previous section, we define the error as the $L^2$ norm of the difference $u_T(\boldsymbol{x}) - u(\boldsymbol{x})$ over the localized domain $K_L = [-L,L]^d$ with $L < R$. Using $v_\infty(\boldsymbol{x}) = u(\boldsymbol{x})$ and the triangle inequality we can split the error as
\begin{align*}
   \| u_T - u \|_{L^2(K_L)} \leq \| u_T - v_T \|_{L^2(K_L)} +  \| v_T -  v_{\infty} \|_{L^2(K_L)}.
\end{align*}
The first term $\lVert u_T-v_T\rVert_{L^2(K_L)}$ we will call the boundary error because it originates from truncating the domain and choosing poor boundary conditions. In section \ref{subsection:Boundaryerror} we will study this error term and show that for well chosen values of $T$ this term decays exponentially. The second term $\lVert v_T(\boldsymbol{x})-v_\infty(\boldsymbol{x})\rVert_{L^2(K_L)}$ we call the modelling error because it appears due to the addition of the regularization term in \eqref{eqn_Modified_Problem}. Notice how this term approaches $0$ as $T\rightarrow \infty$, mirroring how \eqref{eqn_Modified_Problem} approaches \eqref{eq:main_problem_truncated} as $T\rightarrow\infty$. We will study this error term in Section \ref{subsection:modellingerror} and compare the decay to the $d-2$-th order decay of the original elliptic problem.

\subsection{Decay Rates of the Modelling Error} \label{subsection:modellingerror}
In this section we will study the second term  $\| v_T -  v_{\infty} \|_{L^2(K_L)}$ of the error bound. While the regularization term gives an improvement in the boundary error, the modification also creates a new error caused by the modification of the problem, which will be addressed here under two separate assumptions on $g$ in the frequency domain.

\begin{theorem}\label{theorem:mainthrmmodelerr}
    Let $v_T=\int_0^Tv(t,\boldsymbol{x})\,dt$ where $v$ solves \eqref{eqn_InfDomain} with $a\in\mathcal{M}(\alpha,\beta,\mathbb{R}^d)$, and $a \in C^{k,r}(\mathbb{R}^d)$. Furthermore, let $g(\boldsymbol{x})\in L^1(\Omega,1+\lvert \boldsymbol{x}\rvert^{k+1})$ and $\partial^{\bgamma} \hat{g}(\bw)  = 0$ for all $\bgamma \in \mathbb{N}^d$ such that $|\bgamma|\leq k$. Then 
    
    $$\lVert v_T-v_\infty\rVert_{L^2(K_L)}\leq C L^{d/2} T^{-(d + k+ 3)/2},$$
where $C$ is a constant depending on $a$ and $d$ but independent of $L$ and $T$.
\end{theorem}

\begin{theorem}\label{theorem:mainthrmmodelerrexp}
    Let $v_T=\int_0^Tv(t,\boldsymbol{x})\,dt$ where $v$ solves \eqref{eqn_InfDomain} with $a\in\mathcal{M}(\alpha,\beta,\mathbb{R}^d)$. Furthermore, let $g \in L^2(\mathbb{R}^d)$, and $\Hat{g}(\bw)=0, \forall \bw$ such that $ |\bw| \leq\omega_0 $. Then 
    
    $$\lVert v_T-v_\infty \rVert_{L^2(K_L)}\leq C e^{-2\alpha\omega_0^2T}\lVert g\rVert_{L^2(\mathbb{R}^d)},$$
where $C=\frac{1}{2\alpha\omega_0^2}.$
\end{theorem}

It is essential to compare these two theorems with respect to assumptions in the coefficient and the right hand side. 1) Theorems \ref{theorem:mainthrmmodelerr} and \ref{theorem:mainthrmmodelerrexp} provide bounds for the modelling error under different assumptions on the source term $g$; both meaning that $\hat{g}$ has vanishing frequency components in the origin $\bw = 0$. The assumption in Theorem \ref{theorem:mainthrmmodelerr} is weaker, and results in a high order polynomial decay for the modelling error, while the second assumption is somewhat stronger and results in an exponential decay of the modelling error. 2) Theorem \ref{theorem:mainthrmmodelerr} necessitates a high regularity of the coefficient $a$, while Theorem \ref{theorem:mainthrmmodelerrexp} works under minimal assumptions on the coefficient, making it suitable for problems with rapidly varying coefficients.  
\begin{remark}\label{lemma:ghatorigin}
    To motivate the assumptions of vanishing frequency components in the origin, let $u\in L^1(\mathbb{R}^d)$ be a solution to \eqref{eqn_Main_Problem} with constant $a$; say $a=1$. Then 
    \begin{equation}
        -\lvert \bw \rvert^2\hat{u}=\hat{g}.
    \end{equation}
    Since $u\in L^1(\mathbb{R}^d)$ we have $|\hat{u} (\bw)|\leq C, \forall \bw \in \mathbb{R}^d$. This implies that $\hat{g}=\mathcal{O}(\lvert \bw \rvert^2), \forall \bw$, and that in particular $\hat{g}$ vanishes in the origin.
\end{remark}
In order to prove Theorems \ref{theorem:mainthrmmodelerr} and \ref{theorem:mainthrmmodelerrexp} we start by studying the decay rates of solutions of parabolic problems. We have

\begin{theorem} \label{theorem:mainthrmdecayrates}
    Let $v$ solve \eqref{eqn_InfDomain} with $a(\boldsymbol{x})\in\mathcal{M}(\alpha,\beta,\mathbb{R}^d)$. Furthermore, let $g(\boldsymbol{x})\in L^1(\mathbb{R}^d,1+\lvert \boldsymbol{x}\rvert^{k+1})$. Then

    \begin{align*}
        \lVert v(t,\cdot) \rVert_{L^2(\Omega)} \leq C \lvert\Omega\rvert^{d/2} t^{-(d + k+ 1)/2} 
    \end{align*}
where $C$ is a constant depending on $a$ and $d$.
\end{theorem}

\begin{proof}[Proof of Theorem \ref{theorem:mainthrmdecayrates}]
 Using theorem \ref{theorem:gdecomp} and $v = G*g$, the solution to \eqref{eqn_InfDomain} can be represented as follows
\begin{align*}
    v(\boldsymbol{x}) = \sum_{|\bgamma| \leq k} \dfrac{(-1)^{|\bgamma|}}{|\bgamma|!} \int_{\mathbb{R}^d} g(\boldsymbol{x}) \boldsymbol{x}^{\bgamma} \; d\boldsymbol{x} D^{\bgamma} G  + \sum_{|\bgamma|  = k+1} F_{\bgamma} * D^{\bgamma} G.  
\end{align*}
The terms $\int_{\mathbb{R}^d} g(\boldsymbol{x}) \boldsymbol{x}^{\alpha} \; d\boldsymbol{x}$ are connected to the derivatives of the Fourier transform of $g$ at the origin  $\bw = 0$, and will vanish if $\partial^{\bgamma} \hat{g}(0) = 0$, for $|\bgamma| \leq k$. Now, if we can show that $D^{\bgamma} G$ decays in time in some appropriate norm, we will be able to prove a decay rate for the modelling error too. To achieve this, we use the fact that $ a \in C^{k,r}(\mathbb{R}^d)$ the space of Hölder continuous functions with $k$ continuous derivatives, and use Theorem \ref{Eidelman_Survey} to see that
\begin{align*}
    |v(t,\boldsymbol{x}) | \leq \sum_{|\bgamma|  = k+1} \| F_{\bgamma} \|_{L^1(\mathbb{R}^d)} | D^{\bgamma} G |_{\infty} \leq C t^{-(d + k+ 1)/2}. 
\end{align*}
From here we can establish the following interior estimate:
\begin{align*}
    \| v(t,\cdot) \|_{L^2(\Omega)} \leq C \lvert\Omega\rvert^{1/2} t^{-(d + k+ 1)/2}. 
\end{align*}

\end{proof}

\begin{theorem} \label{theorem:mainthrmdecayratesexp}
    Let $v$ solve \eqref{eqn_InfDomain} with $a(\boldsymbol{x})\in\mathcal{M}(\alpha,\beta,\mathbb{R}^d)$. Further let $\Hat{g}(\bw)=0, \forall \bw$ such that $ |\bw| \leq\omega_0 $. Then

    \begin{align*}
        \lVert v(t,\cdot) \rVert_{L^2(\mathbb{R}^d)} \leq e^{-\alpha\omega_0^2t}\lVert g\rVert_{L^2(\mathbb{R}^d)}. 
    \end{align*}
\end{theorem}
\begin{proof}[Proof of Theorem \ref{theorem:mainthrmdecayratesexp}]

To prove this theorem, we will need a variant of the Poincare inequality in $\mathbb{R}^d$ which is of the form
\begin{equation} \label{eq_Poincare_Inf}
        \lVert v(t,\cdot)\rVert^2_{L^2(\mathbb{R}^d)}\leq \frac{1}{\omega_0^2}\lVert\nabla v(t,\cdot)\rVert^2_{L^2(\mathbb{R}^d)}.
\end{equation}
In a nutshell we will prove that indeed this inequality holds under the assumption $\hat{g}(\bw) = 0$ for all $|\bw| \leq w_0$. Assuming this inequality for the moment, we test \eqref{eqn_InfDomain} against $v$, which gives us
\begin{equation*}
    \langle v_t,v\rangle = -\langle a\nabla v,\nabla v\rangle\leq-\alpha\langle \nabla v,\nabla v\rangle.
\end{equation*}
We then have
\begin{equation*}
        \frac{1}{2}\frac{d}{dt}\lVert v\rVert_{L^2(\mathbb{R}^d)}^2 \leq -\alpha \lVert \nabla v\rVert_{L^2(\mathbb{R}^d)}^2 \leq -\alpha\omega_0^2\lVert v\rVert_{L^2(\mathbb{R}^d)}^2
\end{equation*}
This implies that
\begin{equation}
        \lVert v(t,\cdot)\rVert_{L^2(\mathbb{R}^d)}^2\leq e^{-2\alpha\omega_0^2t}\lVert g\rVert_{L^2(\mathbb{R}^d)}^2.
\end{equation}
It remains to prove \eqref{eq_Poincare_Inf}.  We begin by showing $\Hat{g}(\bw)=0, \;\forall |\bw| \leq\omega_0\rightarrow \Hat{v}(t,\bw)=0, \;\forall t\in[0,T),|\bw| \leq\omega_0$. We have 

    \begin{equation*}
        \Hat{v}(t,\bw)=\Hat{G}(t,\bw)\Hat{g}(\bw).
    \end{equation*}
Due to assumptions, we have that the right hand side is clearly $0$ for any $t$ and $|\bw| \leq\omega_0$ as long as $\Hat{G}(t,\bw)$ is bounded. To show that this is the case we will use the following Nash-Aronson estimates
\begin{align*}
    G_1(t,\boldsymbol{x};\boldsymbol{y}) := c t^{-d/2} e^{-\beta \frac{|\boldsymbol{x}-\boldsymbol{y}|^2}{4t}} \leq G(t,\boldsymbol{x};\boldsymbol{y}) \leq G_2(t,\boldsymbol{x};\boldsymbol{y}) := C t^{-d/2} e^{-\alpha \frac{|\boldsymbol{x}-\boldsymbol{y}|^2}{4t}}.   \end{align*}
Whenever $t=0$ we have that the Green's function $G$ is the Dirac delta. The Fourier transform is then constant and bounded. For $t\neq 0$ we can use $\lVert \Hat{G}(t,\cdot)\rVert_{L^\infty(\mathbb{R}^d)}\leq\lVert G(t,\cdot;\boldsymbol{y})\rVert_{L^1(\mathbb{R}^d)}$. By the Nash-Aronson estimate we have that the $L^1$-norm of the Green's function is bounded and then so is the Fourier transform. We then have $\Hat{g}(\bw)=0\;\forall \bw$ such that $ |\bw| \leq\omega_0\rightarrow \Hat{u}(t,\bw)=0\;\forall t\in[0,T)$. We now continue with
 \begin{equation*}
        \lVert \nabla v(t,\cdot)\rVert^2_{L^2(\mathbb{R}^d)}=\int_{|\bw|\geq\omega_0}\lvert\bw\rvert^2\lvert \Hat{v}\rvert^2(\bw)d\bw\geq \omega_0^2\int_{|\bw|\geq\omega_0}\lvert \Hat{v}\rvert^2(\bw)d\bw.
\end{equation*}
Here we used Plancherel's theorem, Fourier transforms of gradients, as well as $\hat{v}(t,\bw)=0$ for $|\bw|\leq\omega_0$. Using Plancherel's theorem again then gives 
\begin{equation*}
        \lVert v(t,\cdot)\rVert^2_{L^2(\mathbb{R}^d)}\leq\frac{1}{\omega_0^2}\lVert \nabla v(t,\cdot)\rVert^2_{L^2(\mathbb{R}^d)}
\end{equation*}
which ends the proof.
\end{proof}

\begin{proof} [Proof of Theorem \ref{theorem:mainthrmmodelerr}]
    The modelling error can be re-written as
    \begin{align*}
        \lVert v_T -  v_{\infty} \rVert_{L^2(K_L)}  = \left\lVert  \int_T^{\infty} v(t,\cdot) \; dt \right\rVert_{L^2(K_L)}\leq \int_{T}^{\infty} \lVert v(t,\cdot)  \rVert_{L^2(K_L)} \; dt.
    \end{align*}
where we applied Minkowskii's inequality. Since $k>1-d$, the integral is bounded and can be rewritten as 
    $$\int_{T}^{\infty} \lVert v(t,\cdot)  \rVert_{L^2(K_L)} \; dt\leq \int_{T}^{\infty} C L^{d/2} t^{-(d + k+ 1)/2}\; dt=C L^{d/2} T^{-(d + k+ 3)/2},$$
where we used Theorem \ref{theorem:mainthrmdecayrates} in the last step. 
\end{proof}
Theorem \ref{theorem:mainthrmmodelerrexp} is proved in the same way by using Theorem \ref{theorem:mainthrmdecayratesexp} instead of Theorem \ref{theorem:mainthrmdecayrates} on the last line. While Theorems \ref{theorem:mainthrmmodelerr} and \ref{theorem:mainthrmmodelerrexp} suggest that the optimal choice for $T$ is $T= \infty$, this would be equivalent to the naive method of simply truncating the domain. We however know that this results in a boundary error of order $d-2$, which in most cases is worse than the modelling error caused by the regularization. As we will see in Section \ref{subsection:Boundaryerror}, a better choice for $T$ is $T\propto \lvert R-L\rvert$.

\subsection{Decay of the Boundary Error} \label{subsection:Boundaryerror}

In this section we study the first term $\lVert u_T-v_T\rVert$ of the error bound which is determined by the difference between the solutions of \eqref{eq:heatequationKR(wdef)} and \eqref{eqn_InfDomain}. 

\begin{definition}[Boundary layer]\label{def:boundarylayer}
    Let us define a sub-domain $K_{\Tilde{R}}\subset K_R$ where $\Tilde{R}$ is the largest integer such that $\Tilde{R}\leq R-\frac{1}{2}$. The boundary layer is defined as the set $\Delta\vcentcolon=K_R\setminus K_{\Tilde{R}}$. Note that $\lvert\Delta\rvert = R^d-\Tilde{R}^d\leq 2dR^{d-1}$. Moreover, we will also use $\Delta_2 \vcentcolon = K_R \setminus K_{2\Tilde{R}-R}$ in the sequel.
\end{definition}


\begin{definition}[Cut-off function]
    A cut-off function on $K_R$ is a function $\rho\in C^\infty(K_R)$ such that

    \begin{equation}
        \rho(\boldsymbol{x})=
        \begin{cases}
            1\text{ in }K_{\Tilde{R}}\\
            0\text{ on }\partial K_R
        \end{cases}
        \text{ and    } \lvert \nabla\rho(\boldsymbol{x})\rvert\leq C_\rho\text{ on }\Delta,
        \label{eq:rhodef}
    \end{equation}
where $K_{\Tilde{R}}$ and $\Delta$ are defined in Definition \ref{def:boundarylayer}.
\end{definition}

\begin{definition}
    Define the function $\theta(t,\boldsymbol{x}) \in L^2([0,\infty);H^1_0(K_R))$ by $\theta(t,\boldsymbol{x})\vcentcolon =w(t,\boldsymbol{x})-\rho(\boldsymbol{x}) v(t,\boldsymbol{x})$ where $w$ and $v$ solve \eqref{eq:heatequationKR(wdef)} and \eqref{eqn_InfDomain}, respectively. Then by linearity, $\theta$ satisfies
    \begin{equation}
    \begin{alignedat}{3}
        &\partial_t \theta(t,\boldsymbol{x})-\nabla\cdot(a(\boldsymbol{x})\nabla\theta(t,\boldsymbol{x})))=F(t,\boldsymbol{x}) &&\quad &&\text{in }K_R\times(0,+\infty),\\
        &\theta(t,\boldsymbol{x}) = 0 &&\quad &&\text{on }\partial K_R\times (0,+\infty),\\
        &\theta(0,\boldsymbol{x})=v(0,\boldsymbol{x})(1-\rho(\boldsymbol{x}))&& \quad &&\text{in }K_R,
        \end{alignedat}
        \label{eq:thetadef}
    \end{equation}
with 
\begin{equation} \label{eq:F_Defn}
    F(t,\boldsymbol{x}) \vcentcolon = -\nabla (1-\rho(\boldsymbol{x}))\cdot a(\boldsymbol{x})\nabla v(t,\boldsymbol{x})-\nabla\cdot[a(\boldsymbol{x})v(t,\boldsymbol{x})\nabla(1-\rho(\boldsymbol{x}))].
\end{equation}
\end{definition}

\begin{definition}
    The Green's function $G(t,\boldsymbol{x};\boldsymbol{y})$ for \eqref{eq:thetadef} is defined as
    \begin{equation}
    \begin{alignedat}{3}
        &\partial_t G(t,\boldsymbol{x};\boldsymbol{y})-\nabla\cdot(a(\boldsymbol{x})\nabla G(t,\boldsymbol{x};\boldsymbol{y}))=0 &&\quad &&\text{in }K_R\times(0,+\infty),\\
        &G(t,\boldsymbol{x};\boldsymbol{y}) = 0 &&\quad &&\text{on }\partial K_R\times (0,+\infty),\\
        &G(0,\boldsymbol{x};\boldsymbol{y})=\delta_{\by}(\boldsymbol{x})&& \quad &&\text{in }K_R,
        \end{alignedat}
        \label{eq:fundsoldef}
    \end{equation}
    where $\delta_{\by}$ is the Dirac delta distribution centered at $\by$.
\end{definition}

\begin{proposition} [Theorem 9 in \cite{Aronson_68}]
    Assume that $a \in \mathcal{M}(\alpha,\beta,K_R)$ and $F \in L^{s}((0,T);L^p(K_{R}))$, with $1<p,s \leq \infty$, such that $ \frac{d}{2p} + \frac{1}{s} \leq 1$. Then the following representation holds  
    \begin{equation}
        \theta(t,\boldsymbol{x}) = \int_{K_R}G(t,\boldsymbol{x};\boldsymbol{y})v(0,\boldsymbol{y})(1-\rho(\boldsymbol{y}))d\boldsymbol{y}+\int_{K_R}\int_0^tG(t-s,\boldsymbol{x};\boldsymbol{y})F(s,\boldsymbol{y})dsd\boldsymbol{y}
    \end{equation}
    for all $t>0$ where $\theta(t,\boldsymbol{x})$ is the weak solution to \eqref{eq:thetadef}, and $G$ is the Green's function solving \eqref{eq:fundsoldef}.
\end{proposition}

\begin{proposition} [See \cite{Aronson_68,F_O_Porper_1984}]
    Let $G(t,\boldsymbol{x};\boldsymbol{y})$ be the Green's function \eqref{eq:fundsoldef} in $ K_R\times(0,+\infty)$ for coefficients $a(\boldsymbol{x})\in\mathcal{M}(\alpha,\beta,K_R)$. Then the Nash-Aronson estimate
    \begin{equation}
        C_1 t^{-d/2} e^{-m \frac{|\boldsymbol{x}-\boldsymbol{y}|^2}{4t}} \leq G(t,\boldsymbol{x};\boldsymbol{y}) \leq C_2 t^{-d/2} e^{-\nu \frac{|\boldsymbol{x}-\boldsymbol{y}|^2}{4t}}
    \end{equation}
    holds for constants $C_1,C_2, m, \nu$ depending only on $\alpha, \beta$, and $d$.
\end{proposition}
With all the necessary definitions we may now state the key theorem for this section.
\begin{theorem}\label{theorem:boundaryerrorbound}
    Let $a\in\mathcal{M}(\alpha,\beta,\mathbb{R}^d)$, and $v, \nabla v \in L^2((0,T);L^2(K_R))$. Moreover, let $u_T(\boldsymbol{x})=\int_0^Tw(t,\boldsymbol{x})dt$ and $v_T(\boldsymbol{x})=\int_0^Tv(t,\boldsymbol{x})dt$, where $v$ and $w$ are given by \eqref{eqn_InfDomain} and \eqref{eq:heatequationKR(wdef)} respectively, and $K_L\subset K_{\Tilde{R}}$. Then
    \begin{equation}\begin{split}
        \lVert u_T-v_T\rVert_{L^2(K_L)}\leq &\lvert K_L\rvert^{1/2}\lVert g\rVert_{L^2(\mathbb{R}^d)}\left(C_1+C_2\frac{T}{\lvert R-L\rvert}+C_3\frac{T^3}{\lvert R -L\rvert}\right)\\&\frac{1}{\sqrt{c}\lvert R-L\rvert}T^{\frac{2-d}{2}}e^{-c\frac{\lvert R-L\rvert^2}{T}},
    \end{split}\end{equation}
    where $C_1 = C\lvert\Delta\rvert^\frac{1}{2}$, $C_2 = \frac{C\lvert \Delta\lvert^\frac{1}{2}b}{2\sqrt{c\alpha}}$, $C_3 = \frac{2M\beta}{\alpha}\frac{C\lvert \Delta_2\lvert^\frac{1}{2}b}{\sqrt{2c}}$, and $\Delta, \Delta_2$ are defined as in Definition \ref{def:boundarylayer}, and $b,M,c,C$ are the same constansts as in Lemma \ref{Lem:G_Estimates}. 
\end{theorem}
The proof of Theorem \ref{theorem:boundaryerrorbound} requires some intermediate lemmas. 

\begin{lemma} 
    Let $G(t,\boldsymbol{x};\boldsymbol{y})$ be the Green's function of \eqref{eq:thetadef} for coefficients $a(\boldsymbol{x})\in \mathcal{M}(\alpha,\beta,K_R)$. Let $\Delta \vcentcolon = K_R\setminus K_{\Tilde{R}}$ and $\Delta_2= K_R\setminus K_{2\Tilde{R}-R}$. Then, for all $\boldsymbol{y}\in K_R$ such that $dist (\boldsymbol{y},\Delta_2)>0$, there exists a positive constant $M = O(1)$ such that

    \begin{equation}
        \lVert \nabla G(\cdot,\cdot;\boldsymbol{y})\rVert_{L^2((0,T)\times \Delta)}\leq \frac{2M\beta}{\alpha}\lVert G(\cdot,\cdot;\boldsymbol{y})\rVert_{L^2((0,T)\times \Delta_2)}.
    \end{equation}
\end{lemma}

\begin{proof}
    See \cite{abdulle2021parabolic}
\end{proof}

\begin{lemma}
    \label{lemma:vbound}
    Let $v$ solve equation \eqref{eqn_InfDomain} with $a(\boldsymbol{x})\in\mathcal{M}(\alpha,\beta,\mathbb{R}^d)$ and $v(t,\boldsymbol{x})\rightarrow 0$, as $\lvert \boldsymbol{x}\rvert\rightarrow \infty$ for all $t\geq0$, and $g(\boldsymbol{x})\in L^2(\mathbb{R}^d)$. Then
    
    \begin{equation}
        \lVert v\rVert_{L^2((0,t)\times \Delta)}\leq \lVert v\rVert_{L^2((0,t)\times \mathbb{R}^d)}\leq \sqrt{t}\lVert g\rVert_{L^2(\mathbb{R}^d)}
    \end{equation}
\end{lemma}

\begin{proof}
    The first inequality follows directly from $\Delta \subset \mathbb{R}^d$. For the second inequality we begin by testing equation \eqref{eqn_InfDomain} against the test function $v$
    
    \begin{equation}
        \langle v_t,v\rangle_{L^2(\mathbb{R}^d)} -\langle \nabla \cdot(a\nabla v),v\rangle_{L^2(\mathbb{R}^d)} = 0.
    \end{equation}
    With $\langle v,v_t\rangle = \frac{1}{2}\frac{d}{dt}\lVert v\rVert^2$ we have
    \begin{equation}
        \frac{1}{2}\frac{d}{dt}\lVert v(t,\cdot)\rVert^2_{L^2(\mathbb{R}^d)} +\int_{\mathbb{R}^d}a(\boldsymbol{x})\lvert \nabla v(t,\boldsymbol{x}) \rvert^2 dx = 0.
    \end{equation}
    Because we have $a>0$ we get
    \begin{equation}
        \frac{1}{2}\frac{d}{dt}\lVert v(t,\cdot)\rVert^2_{L^2(\mathbb{R}^d)} =-\int_{\mathbb{R}^d}a(\boldsymbol{x})\lvert \nabla v(t,\boldsymbol{x}) \rvert^2 dx\leq 0.
    \end{equation}
    This implies that $\lVert v \rVert$ is monotone decreasing which gives us the inequality
    \begin{equation}
        \lVert v(t,\cdot) \rVert^2_{L^2(\mathbb{R}^d)}\leq \lVert v(0,\cdot) \rVert^2_{L^2(\mathbb{R}^d)}=\lVert g \rVert^2_{L^2(\mathbb{R}^d)}
    \end{equation}
    Let us now consider 
    \begin{equation}
        \lVert v \rVert^2_{L^2((0,t)\times\mathbb{R}^d)}=\int_0^t\int_{\mathbb{R}^d}v(s,\boldsymbol{x})^2d\boldsymbol{x}ds\leq \int_0^t\int_{\mathbb{R}^d}g(\boldsymbol{x})^2 d\boldsymbol{x}ds = t\lVert g\rVert^2_{L^2(\mathbb{R}^d)}.
    \end{equation}
    Rewriting as
    \begin{equation}
        \lVert v \rVert_{L^2((0,t)\times\mathbb{R}^d)}\leq \sqrt{t}\lVert g\rVert_{L^2(\mathbb{R}^d)}
    \end{equation}
    then concludes the proof.
\end{proof}

\begin{lemma}
    \label{lemma:gradvbound}
    Let $v$ solve equation \eqref{eqn_InfDomain} with coefficients $a(\boldsymbol{x})\in\mathcal{M}(\alpha,\beta,\mathbb{R}^d)$ and $\lim_{\lvert \boldsymbol{x}\rvert\rightarrow \infty} v(\boldsymbol{x})\rightarrow 0$ and $g\in L^2(\mathbb{R}^d)$. Then
    
    \begin{equation}
        \lVert \nabla v\rVert_{L^2((0,t)\times \Delta)}\leq \lVert \nabla v\rVert_{L^2((0,t)\times \mathbb{R}^d)}\leq \frac{1}{\sqrt{2\alpha}}\lVert g\rVert_{L^2(\mathbb{R}^d)}
    \end{equation}
    
\end{lemma}

\begin{proof}
    The first inequality follows directly from $\Delta \subset \mathbb{R}^d$. For the second inequality we begin with testing equation \eqref{eqn_InfDomain} with $v$ as the test function
    \begin{equation}
        \langle v_t,v\rangle_{L^2(\mathbb{R}^d)} -\langle \nabla \cdot(a\nabla v),v\rangle_{L^2(\mathbb{R}^d)} = 0.
    \end{equation}
    With $\langle v,v_t\rangle = \frac{1}{2}\frac{d}{dt}\lVert v(t,\cdot)\rVert_{L^2(\mathbb{R}^d}^2$ we have
    \begin{equation}
        \frac{1}{2}\frac{d}{dt}\lVert v(t,\cdot)\rVert^2_{L^2(\mathbb{R}^d)} +\int_{\mathbb{R}^d}a(\boldsymbol{x})\lvert \nabla v(t,\boldsymbol{x}) \rvert^2 d\boldsymbol{x} = 0.
    \end{equation}
    Integrating in time on $[0,t)$ gives 
    \begin{equation}
        \frac{1}{2}\lVert v(t,\cdot) \rVert^2_{L^2(\mathbb{R}^d)}-\frac{1}{2}\lVert v(0,\cdot)\rVert^2_{L^2(\mathbb{R}^d)}+\int_0^t\int_{\mathbb{R}^d}a(\boldsymbol{x})\lvert \nabla v\rvert^2 d\boldsymbol{x}dt=0.
    \end{equation}
    We can move the second term to the other side while noting that $v(0,\boldsymbol{x})=g(\boldsymbol{x})$ and note that the first term is non-negative and as such removing it makes the LHS smaller. We get 
    \begin{equation}
        \int_0^t\int_{\mathbb{R}^d}a(\boldsymbol{x})\lvert \nabla v(t,\boldsymbol{x})\rvert^2 d\boldsymbol{x}dt\leq\frac{1}{2}\lVert g\rVert^2_{L^2(\mathbb{R}^d)}.
    \end{equation}
    By assumptions $a$ is bounded below by the constant $\alpha$ and as such we have
    \begin{equation}
        \int_0^t\int_{\mathbb{R}^d}\lvert \nabla v(t,\boldsymbol{x})\rvert^2 d\boldsymbol{x}dt\leq\frac{1}{2\alpha}\lVert g\rVert^2_{L^2(\mathbb{R}^d)}.
    \end{equation}
    Rewriting as 
    \begin{equation}
        \lVert \nabla v\rVert_{L^2((0,t)\times\mathbb{R}^d)}\leq \frac{1}{\sqrt{2\alpha}}\lVert g\rVert_{L^2(\mathbb{R}^d)}
    \end{equation}
    concludes the proof.
\end{proof}

\begin{lemma} \label{Lem:G_Estimates}

    Let $G(t,\boldsymbol{x};\boldsymbol{y})$ be the Green's function of \eqref{eq:thetadef} in $K_R\times(0,+\infty)$ with coefficients $a(\boldsymbol{x})\in\mathcal{M}(\alpha,\beta,K_R)$. Then there exists a positive constant $M = O(1)$, and arbitrary constant $b\geq\frac{\lvert R-L\rvert}{\lvert 2\Tilde{R}-R-L\rvert}$ with $\Tilde{R}$ defined as in Definition \ref{def:boundarylayer}. Then the following inequalities on $G$ hold

    \begin{align}
    \lVert G(t,\cdot;\boldsymbol{x})\rVert_{L^2(\Delta)}&\leq \frac{C\lvert\Delta\rvert^\frac{1}{2}}{t^{\frac{d}{2}}}e^{-c\frac{\lvert R-L\rvert^2}{t}}.\label{eq:Gbound1}\\
    \lVert G(\cdot,\cdot;\boldsymbol{x})\rVert_{L^2((0,t)\times\Delta)}&\leq \frac{C\lvert \Delta\lvert^\frac{1}{2}b}{\sqrt{2\nu}\lvert R-L\rvert}\frac{1}{t^{\frac{d}{2}-1}}e^{-c\frac{\lvert R-L\rvert^2}{t}}.\label{eq:Gbound2}\\
    \lVert \nabla G(\cdot,\cdot;\boldsymbol{x})\rVert_{L^2((0,t)\times\Delta)}&\leq \frac{2M\beta}{\alpha} \frac{C\lvert \Delta_2\lvert^\frac{1}{2}b}{\sqrt{2\nu}\lvert R-L\rvert}\frac{1}{t^{\frac{d}{2}-1}}e^{-c\frac{\lvert R-L\rvert^2}{t}},\label{eq:Gbound3}
    \end{align}
    where $0<c\leq \nu$, $0<t<2\nu |\tilde{R} - L|^2$, and $C>0$ is a constant depending only on $d,\alpha$, and $\beta$. 
\end{lemma}

\begin{proof}
    See \cite{abdulle2021parabolic}.
\end{proof}

We may now provide an upper bound for the spatial $L^2$-norm of $\theta$.

\begin{lemma}\label{lemma:thetalinfupperbound}
    Let $\theta(t,\boldsymbol{x})\vcentcolon=w(t,\boldsymbol{x})-\rho(\boldsymbol{x}) v(t,\boldsymbol{x})$ be the solution to \eqref{eq:thetadef}. Let $w$ and $v$ be the solution of \eqref{eq:heatequationKR(wdef)} and \eqref{eqn_InfDomain} respectively with $a(\boldsymbol{x})\in\mathcal{M}(\alpha,\beta,K_R)$ and $\lim_{\lvert \boldsymbol{x}\rvert\rightarrow \infty} v(\boldsymbol{x})\rightarrow 0$. Then

    \begin{equation}
        \lVert\theta(t,\cdot)\rVert_{L^\infty(K_L)}\leq\left(C_1+C_2\frac{t}{\lvert R-L\rvert}+C_3\frac{t^3}{\lvert R -L\rvert}\right)\frac{1}{t^{\frac{d}{2}}}e^{-c\frac{\lvert R-L\rvert^2}{t}}\lVert g\rVert_{L^2(\mathbb{R}^d)}
    \end{equation}     
    where $C_1 = C\lvert\Delta\rvert^\frac{1}{2}$, $C_2 = \frac{C\lvert \Delta\lvert^\frac{1}{2}b}{2\sqrt{\nu\alpha}}$, and $C_3 = \frac{2M\beta}{\alpha}\frac{C\lvert \Delta_2\lvert^\frac{1}{2}b}{\sqrt{2\nu}}$, where $\Delta$, and $\Delta_2$ are defined  in Definition \ref{def:boundarylayer}, and $b,c,C,\nu$ are the same constants as in Lemma \ref{Lem:G_Estimates}.
\end{lemma}

\begin{proof}
Using equations \eqref{eq:thetadef} and \eqref{eq:F_Defn},
we have 

\begin{equation}
\begin{split}
    \lvert \theta(t,\boldsymbol{x})\rvert \leq &
    \underbrace{\lVert G(t,\cdot;\boldsymbol{x})\rVert_{L^2(\Delta)}\lVert g\rVert_{L^2(\Delta)}}_{T_1}\\
    &+\underbrace{C_\rho\lVert G(\cdot,\cdot;\boldsymbol{x})\rVert_{L^2((0,t)\times\Delta)}\lVert \nabla v\rVert_{L^2((0,t)\times \Delta)}}_{T_2}\\
    &+\underbrace{C_\rho\lVert\nabla G(\cdot,\cdot;\boldsymbol{x})\rVert_{L^2((0,t)\times\Delta)}\lVert v\rVert_{L^2((0,t)\times\Delta)}}_{T_3}. \label{eq:thetabound}
\end{split}
\end{equation}
We will consider these terms separately. We begin with $T_1$. Using equation \eqref{eq:Gbound1} we have
\begin{equation}
    T_1=\lVert G(t,\cdot;\boldsymbol{x})\rVert_{L^2(\Delta)}\lVert g\rVert_{L^2(\Delta)}\leq \frac{C\lvert\Delta\rvert^\frac{1}{2}}{t^{\frac{d}{2}}}e^{-c\frac{\lvert R-L\rvert^2}{t}}\lVert g\rVert_{L^2(\mathbb{R}^d)}. \label{eq:T1bound} 
\end{equation}
We can find an upper bound for $T_2$ by using lemma \ref{lemma:gradvbound} and equation \eqref{eq:Gbound2}. This gives us
\begin{equation}
\begin{split}
    T_2 = &C_\rho\lVert G(\cdot,\cdot;\boldsymbol{x})\rVert_{L^2((0,t)\times\Delta)}\lVert \nabla v\rVert_{L^2((0,t)\times \Delta)} \leq\\
    \leq&\frac{C\lvert \Delta\lvert^\frac{1}{2}b}{2\sqrt{\nu\alpha}\lvert R-L\rvert}\frac{1}{t^{\frac{d}{2}-1}}e^{-c\frac{\lvert R-L\rvert^2}{t}}\lVert g\rVert_{L^2(\mathbb{R}^d)} \label{eq:T2bound}
\end{split}
\end{equation}
For $T_3$ we can use lemma \ref{lemma:vbound} and equation \eqref{eq:Gbound3}, giving us
\begin{equation}
\begin{split}
    T_3 = &C_\rho\lVert \nabla G(\cdot,\cdot;\boldsymbol{x})\rVert_{L^2((0,t)\times\Delta)}\lVert v\rVert_{L^2((0,t)\times \Delta)} \leq\\
    \leq&\frac{2M\beta}{\alpha}\frac{C\lvert \Delta\lvert^\frac{1}{2}b}{\sqrt{2\nu}\lvert R-L\rvert}\frac{1}{t^{\frac{1}{2}(d-3)}}e^{-c\frac{\lvert R-L\rvert^2}{t}}\lVert g\rVert_{L^2(\mathbb{R}^d)} \label{eq:T3bound}
\end{split}
\end{equation}
Combining equation \eqref{eq:thetabound} with equation \eqref{eq:T1bound}, \eqref{eq:T2bound}, and \eqref{eq:T3bound} gives
\begin{equation}
    \begin{split}
        \lvert \theta(t,\boldsymbol{x})\rvert \leq &\frac{C\lvert\Delta\rvert^\frac{1}{2}}{t^{\frac{d}{2}}}e^{-c\frac{\lvert R-L\rvert^2}{t}}\lVert g\rVert_{L^2(\mathbb{R}^d)}+\\
        +&\frac{C\lvert \Delta\lvert^\frac{1}{2}b}{2\sqrt{\nu\alpha}\lvert R-L\rvert}\frac{1}{t^{\frac{d}{2}-1}}e^{-c\frac{\lvert R-L\rvert^2}{t}}\lVert g\rVert_{L^2(\mathbb{R}^d)}+\\
        +&\frac{2M\beta}{\alpha}\frac{C\lvert \Delta\lvert^\frac{1}{2}b}{\sqrt{2\nu}\lvert R-L\rvert}\frac{1}{t^{\frac{1}{2}(d-3)}}e^{-c\frac{\lvert R-L\rvert^2}{t}}\lVert g\rVert_{L^2(\mathbb{R}^d)}=\\
        =&\left(\frac{C\lvert\Delta\rvert^\frac{1}{2}}{t^{\frac{d}{2}}}+\frac{C\lvert \Delta\lvert^\frac{1}{2}b}{2\sqrt{\nu\alpha}\lvert R-L\rvert}\frac{1}{t^{\frac{d}{2}-1}}+\frac{2M\beta}{\alpha}\frac{C\lvert \Delta_2\lvert^\frac{1}{2}b}{\sqrt{2\nu}\lvert R-L\rvert}\frac{1}{t^{\frac{1}{2}(d-3)}}\right)e^{-c\frac{\lvert R-L\rvert^2}{t}}\lVert g\rVert_{L^2(\mathbb{R}^d)}=\\
        =&\left(C_1+C_2\frac{t}{\lvert R-L\rvert}+C_3\frac{t^3}{\lvert R -L\rvert}\right)\frac{1}{t^{\frac{d}{2}}}e^{-c\frac{\lvert R-L\rvert^2}{t}}\lVert g\rVert_{L^2(\mathbb{R}^d)}
    \end{split}
\end{equation}
where $C_1 = C\lvert\Delta\rvert^\frac{1}{2}$, $C_2 = \frac{C\lvert \Delta\lvert^\frac{1}{2}b}{2\sqrt{\nu\alpha}}$, and $C_3 = \frac{2M\beta}{\alpha}\frac{C\lvert \Delta_2\lvert^\frac{1}{2}b}{\sqrt{2\nu}}$.

\end{proof}
Now, we prove Theorem \ref{theorem:boundaryerrorbound}.

\begin{proof}[Proof of Theorem \ref{theorem:boundaryerrorbound}]
   We begin with
    \begin{equation}
    \begin{split}
        \lVert u_T-v_T\rVert^2=&\int_{K_L}(u_T(\boldsymbol{x})-v_T(\boldsymbol{x}))^2dx = \\
        =&\int_{K_L}\left(\int_0^Tw(t,\boldsymbol{x})dt-\int_0^Tv(t,\boldsymbol{x})dt\right)^2dx=\\
        =&\int_{K_L}\left(\int_0^Tw(t,\boldsymbol{x})-\rho(\bx) v(t,\boldsymbol{x})dt\right)^2dx= \\
        =&\int_{K_L}\left(\int_0^T\theta(t,\bx) dt\right)^2dx \leq \left(  \int_{0}^{T}  \left( \int_{K_{L}} |\theta(t,\bx)|^2 \; d\bx  \right)^{1/2} \; dt \right)^2,
        \label{eq:thetaproof1}
    \end{split}
    \end{equation}
    where we employed the Minkowski's inequality in the last line. Therefore, by lemma \ref{lemma:thetalinfupperbound} we then have
    \begin{equation}
    \begin{split}
        \lVert u_T-v_T\rVert \leq& \int_{0}^{T}  \left( \int_{K_{L}} |\theta(t,\bx)|^2 \; d\bx  \right)^{1/2} \; dt \\ \leq&
        |K_L|^{1/2} \int_{0}^{T} \| \theta(t,\cdot) \|_{L^{\infty}(K_L)} \; dt \\
        \leq& |K_L|^{1/2} \lVert g\rVert_{L^2(\mathbb{R}^d)} \int_0^T\left(C_1+C_2\frac{t}{\lvert R-L\rvert}+C_3\frac{t^3}{\lvert R -L\rvert}\right) \frac{1}{t^{{d/2}}}e^{-c\frac{\lvert R-L\rvert^2}{t}}dt
    \end{split}
    \end{equation}
    Next we use the following inequality
    $$\int_{0}^{T} \dfrac{1}{s^d} e^{-c\frac{|R-L|^2}{s}} \leq \dfrac{C}{c|R-L|^2} T^{2-d} e^{-c \frac{|R-L|^2}{T}},$$ whenever $T< c |R-L|^2$ to see that
\begin{equation}
    \begin{split}
        \lVert u_T-v_T\rVert_{L^2(K_L)} \leq& \lvert K_L\rvert^{1/2}\lVert g\rVert_{L^2(\mathbb{R}^d)}\left(C_1+C_2\frac{T}{\lvert R-L\rvert}+C_3\frac{T^3}{\lvert R -L\rvert}\right) \dfrac{T^{2-d/2}}{c |R-L|^2} e^{-c \frac{|R-L|^2}{T}}
    \end{split}
    \end{equation}
\end{proof}
If we choose $T\propto \lvert R-L\rvert$ we end up with exponential decay with respect to R in the boundary error. Clearly this is an improvement over the polynomial decay of the boundary error of the unmodified problem.



\section{Numerical Results}
\label{sec:NumericalResults}
In this section we aim to demonstrate the rate of convergence and computational efficiency of the regularized method and compare it with the naive approach where no regularization is introduced. The simulations aim at demonstrating the following 

\begin{enumerate}
    \item A comparison of the total $L^2$-errors between the methods.
    \item Decay of the boundary error.
    \item Decay of the modelling error.
    \item The error when the coefficient is quasi-periodic.
\end{enumerate}
These experiments aim to validate theoretical predictions for various settings beyond the periodic setting.

In all experiments, we define the standard method as the solution to \eqref{eq:main_problem_truncated} and the regularized method as the solution to \eqref{eqn_Modified_Problem}. The error is defined as the $L^2$-norm over $K_L=K_{\frac{2}{3}R}$ of the difference between the computed and true solution to \eqref{eqn_Main_Problem}. For the regularized method, choosing $L<R$ is necessary for attaining exponential convergence in the boundary error, as is evident from theorem \ref{theorem:boundaryerrorbound}. 

Note that the exponential regularization term $\left[e^{-T\mathcal{A}}g\right](\bx)$ corresponds to the solution of the parabolic problem \eqref{eq:heatequationKR(wdef)} at time $T$. The most straightforward way to do so is to discretize and solve the parabolic problem \eqref{eq:heatequationKR(wdef)}. If we instead consider a discretization $\mathcal{A}_h$ of $\mathcal{A}$, we may compute the action of the exponential on $g$ directly.  For finite elements we have $\mathcal{A}_h=M^{-1}A$ where $M$ is the mass matrix and $A$ is the stiffness matrix associated with the operator $\nabla\cdot (a(\boldsymbol{x})\nabla v(t,\boldsymbol{x}))$ equipped with Dirichlet boundary conditions. To reduce computational cost of the inversion of the mass matrix we employ mass lumping. The action of the exponential operator on $g$ is then computed using the method described in \cite{al2011computing}.

The parameter $T$ is set as $T=\frac{1}{2\pi\sqrt{\alpha\beta}}\lvert R-L\rvert$ for all experiments except experiment 2, where it is held constant to maintain a fixed modelling error as $R$ varies. Here, $\alpha$ and $\beta$ are the ellipticity and continuity constants of the coefficient $a$. The numerical implementation uses FEniCS, \cite{FeniCsBook}, with linear Lagrange (CG) basis functions and homogenuous Dirichlet boundary conditions. The spatial discretization uses $140R$ gridpoints in each direction (rounded) for all experiments except experiment 1, where we use $40R$ gridpoints. All linear systems are solved using the generalized minimal residual method with incomplete LU preconditioning with PETSc as a backend.

\subsection{Total Error}

Here we present numerical results comparing the total error for each method. For this experiment we consider the 3-dimensional case with our coefficient $a$ given by
\begin{equation}
    a(\boldsymbol{x})=e^{\frac{1}{x_1^2+x_2^2+x_3^2+1}}\mathbb{I}_3
\end{equation}
where $\mathbb{I}_3$ denotes the $3\times 3$ identity matrix. To enable comparison with an exact solution, we construct a source term $g(\boldsymbol{x})$ from a chosen solution $u(\boldsymbol{x})$. Specifically, we set
\begin{equation}
    u(\boldsymbol{x}) = (\sin(2\pi x_1)+\sin(2\pi x_2)+\sin(2\pi x_3))(e^{\frac{1}{x_1^2+x_2^2+x_3^2+1}}-1),
\end{equation}
which yields the source term
\begin{equation}
    g(\boldsymbol{x})=-\nabla\cdot(a(\boldsymbol{x})\nabla u(\boldsymbol{x})).
\end{equation}
Because we construct our source term from a desired solution $u$, we can compare our numerical solution to the exact solution of \eqref{eqn_Main_Problem}. Figure \ref{fig:3dRvserr} presents the computed error as a function of $R$. We observe that the error of the regularized method decays significantly faster with increasing $R$ when compared to the standard method, in line with theory. Further, not only is the rate of decay improved significantly, the error itself is much smaller even for rather moderate values of $R$.

Figure \ref{fig:3druntimevserr} shows the computed error as a function of runtime. Despite the additional computational cost of calculating the matrix exponential, the regularized method achieves a given error tolerance more quickly than the standard method due to the superior convergence rate.

\begin{figure}[H]
    \centering
    \includegraphics[width = 0.8\textwidth]{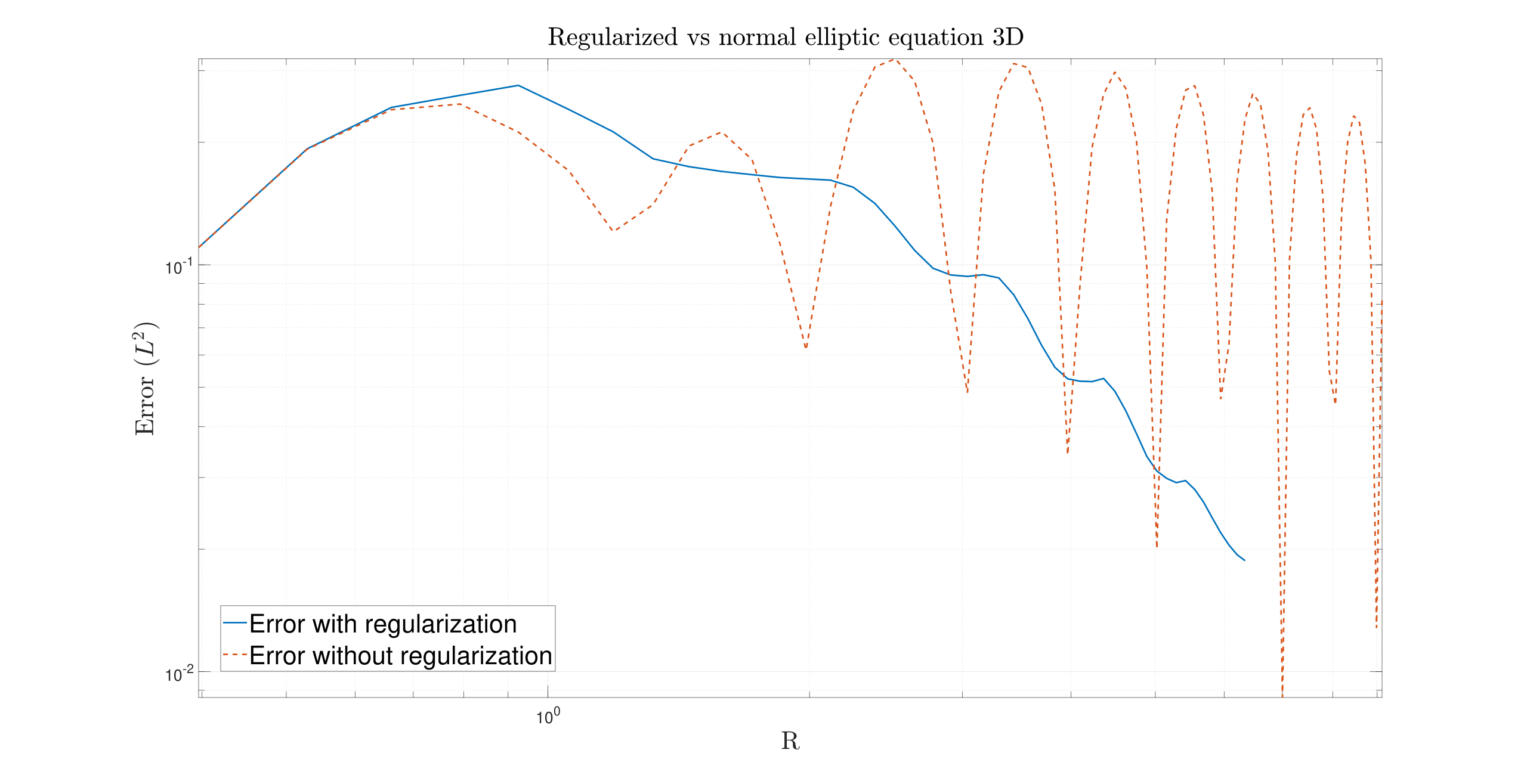}
    \caption{Error convergence with respect to domain size $R$ for the 3D problem. The regularized method shows superior asymptotic behavior compared to the standard method}
    \label{fig:3dRvserr}
\end{figure}

\begin{figure}[H]
    \centering
    \includegraphics[width = 0.8\textwidth]{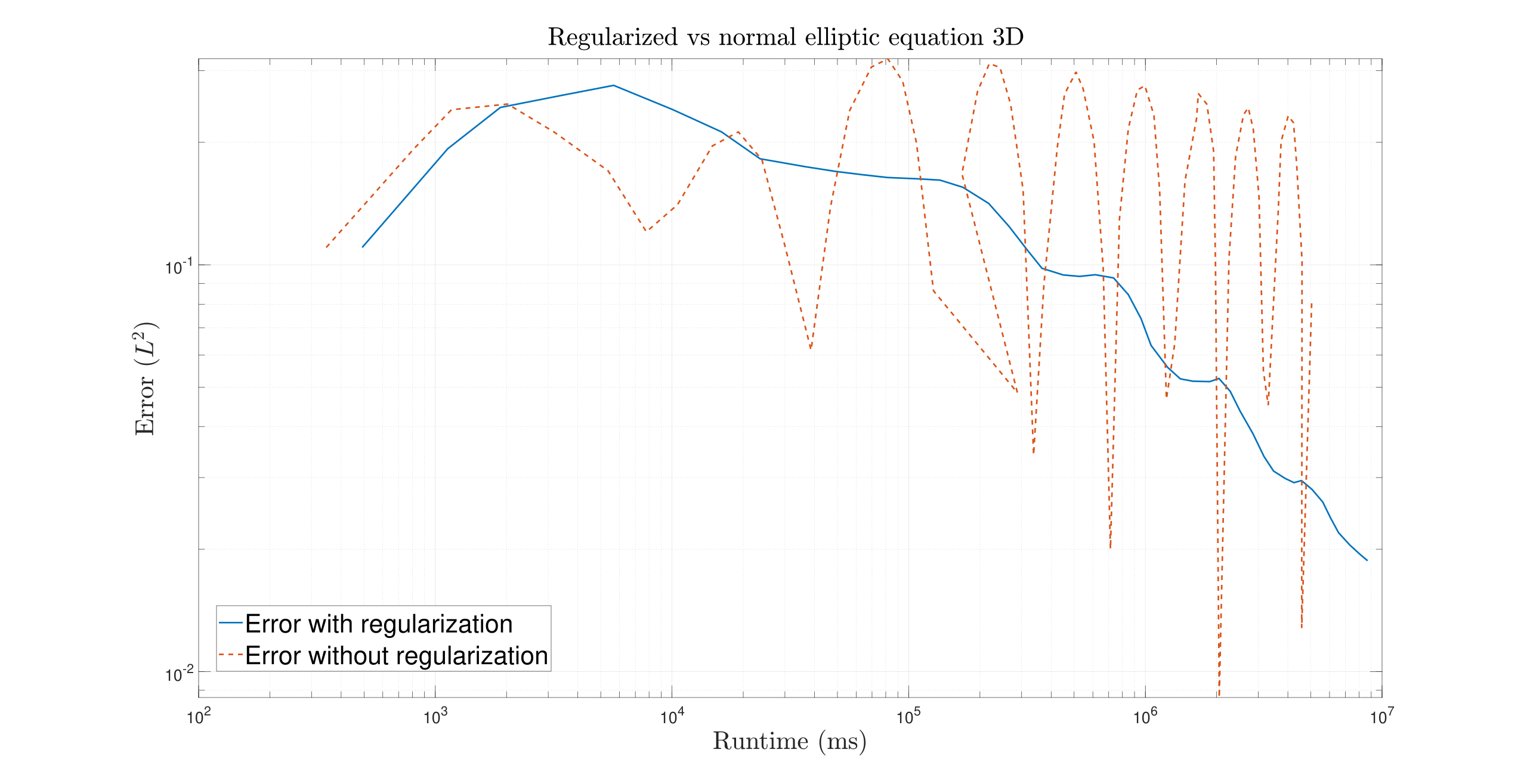}
    \caption{Error versus computational runtime for the 3D problem. The regularized method shows improved efficiency despite the higher computational cost for a given problem size}
    \label{fig:3druntimevserr}
\end{figure}

\subsection{Boundary Error}

We now examine the boundary error for both methods in a two-dimensional setting. The coefficient matrix $a$ is defined as
\begin{equation}
    a(\boldsymbol{x})=e^{\frac{1}{x_1^2+x_2^2+1}}\mathbb{I}_2,
\end{equation}
where $\mathbb{I}_2$ is the $2\times 2$ identity matrix. We choose the solution
\begin{equation}
    u(\boldsymbol{x}) = \sin(2\pi x_1)(e^{\frac{1}{x_1^2+x_2^2+1}}-1),
\end{equation}
which yields the source term
\begin{equation}
    g(\boldsymbol{x})=-\nabla\cdot(a(\boldsymbol{x})\nabla u(\boldsymbol{x})).
\end{equation}
To isolate boundary error effects from modelling error, we use a high-accuracy solution computed via the regularized method with $R=15$ as our reference, rather than the exact solution. A constant $T=\frac{1}{2\pi\sqrt{\alpha\beta}}\lvert R-L\rvert$ with $\frac{3}{2}L=R=15$ is used throughout the experiments to keep the modelling error constant for all runs.

Figure \ref{fig:2dRvboundaryserr} demonstrates the exponential decay of the boundary error for the regularized method, in line with Theorem \ref{theorem:boundaryerrorbound}. In contrast, Figure \ref{fig:2dRvboundaryserrnoreg} shows that the standard method exhibits only first-order decay in $R$ for the boundary error.

\begin{figure}[H]
    \centering
    \includegraphics[width = 0.8\textwidth]{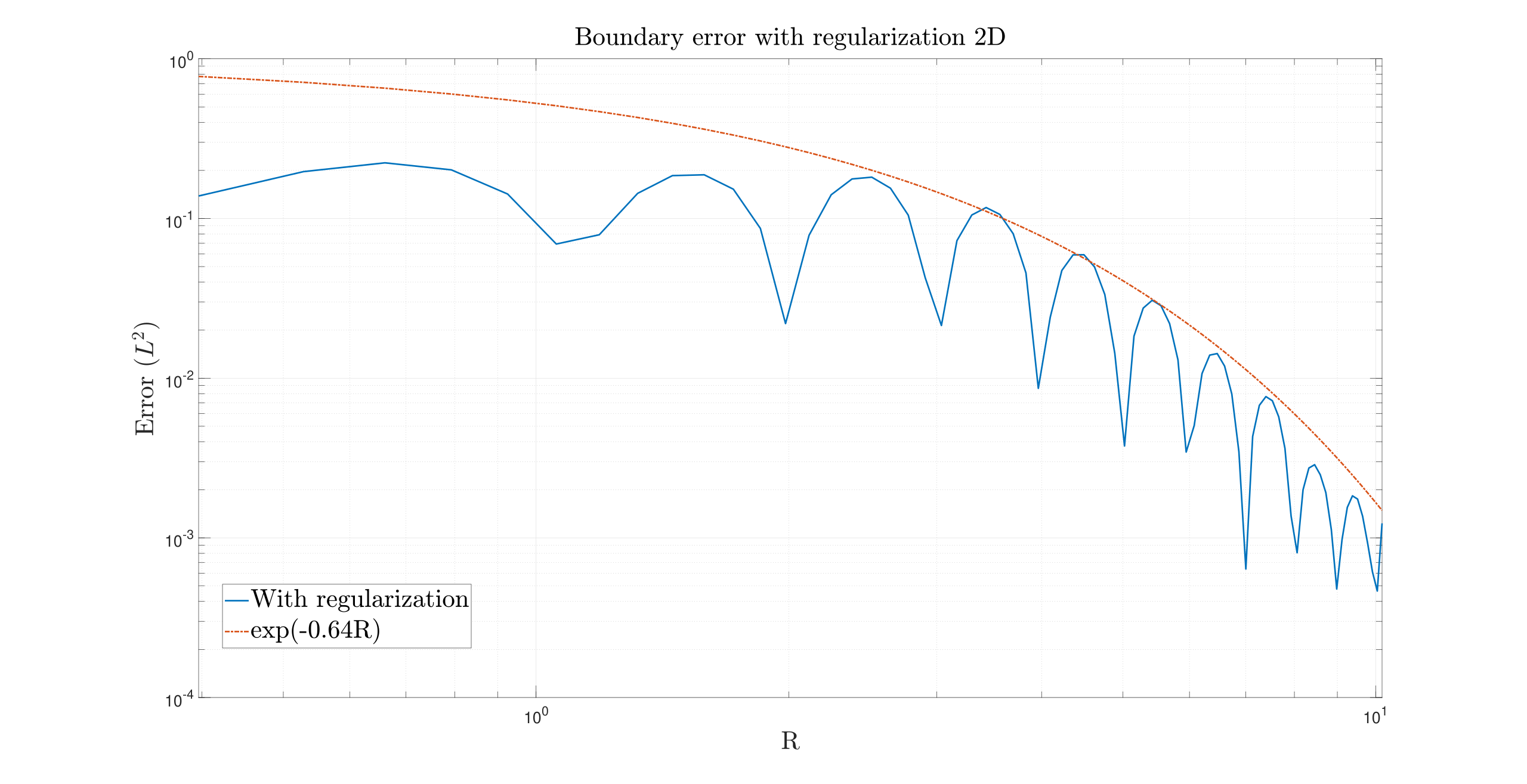}
    \caption{Relation between domain size $R$ and $L^2$-boundary-error for the 2D problem using the regularized method, demonstrating the exponential decay predicted by theory}
    \label{fig:2dRvboundaryserr}
\end{figure}

\begin{figure}[H]
    \centering
    \includegraphics[width = 0.8\textwidth]{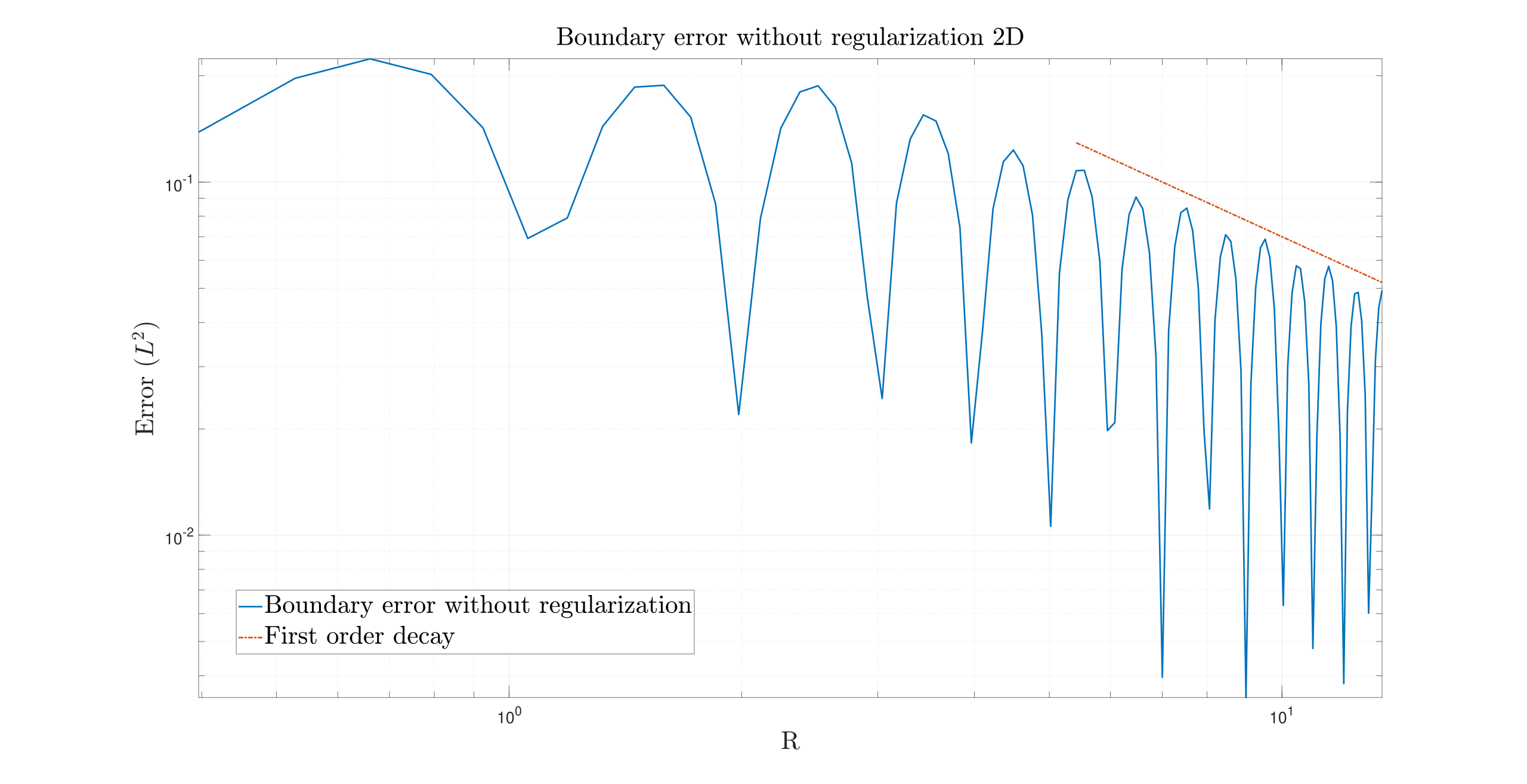}
    \caption{Relation between domain size $R$ and $L^2$-boundary-error for the 2D problem using the standard method, showing first-order decay}
    \label{fig:2dRvboundaryserrnoreg}
\end{figure}

\subsection{Modelling Error}

The boundary error arises from the mismatch between true and computational boundary conditions. By constructing a problem such that the boundary conditions match up, we may avoid a boundary error, allowing us to study the modelling error. We consider a 2-dimensional case with coefficient matrix
\begin{equation}
    a(\boldsymbol{x})=e^{\frac{1}{x_1^2+x_2^2+1}}\mathbb{I}_2,
\end{equation}
where $\mathbb{I}_2$ is the $2\times 2$ identity matrix. The exact solution is chosen as
\begin{equation}
    u(\boldsymbol{x})=\sin(10\pi x_1)\sin(10\pi x_2)(e^{\frac{1}{x_1^2+x_2^2+1}}-1),
\end{equation}
which yields the source term
\begin{equation}
    g(\boldsymbol{x})=-\nabla\cdot(a(\boldsymbol{x})\nabla u(\boldsymbol{x})).
\end{equation}
To eliminate boundary error effects, we consider only values $R=0.2n$ for $n\in \mathbb{N}$. This way, $u(\bx)=0$ on the boundaries, matching up with our homogenuous Dirichlet boundary conditions. Figure \ref{fig:2dRvmodellingerr} presents the results, demonstrating rapid decay of the modelling error with respect to $R$.

It's important to note that the modelling error fundamentally depends on $T$, not $R$. However, since the optimal choice of $T$ is dependent on $R$, the relationship between the modelling error and $R$ remains of significant interest.

\begin{figure}[H]
    \centering
    \includegraphics[width = 0.8\textwidth]{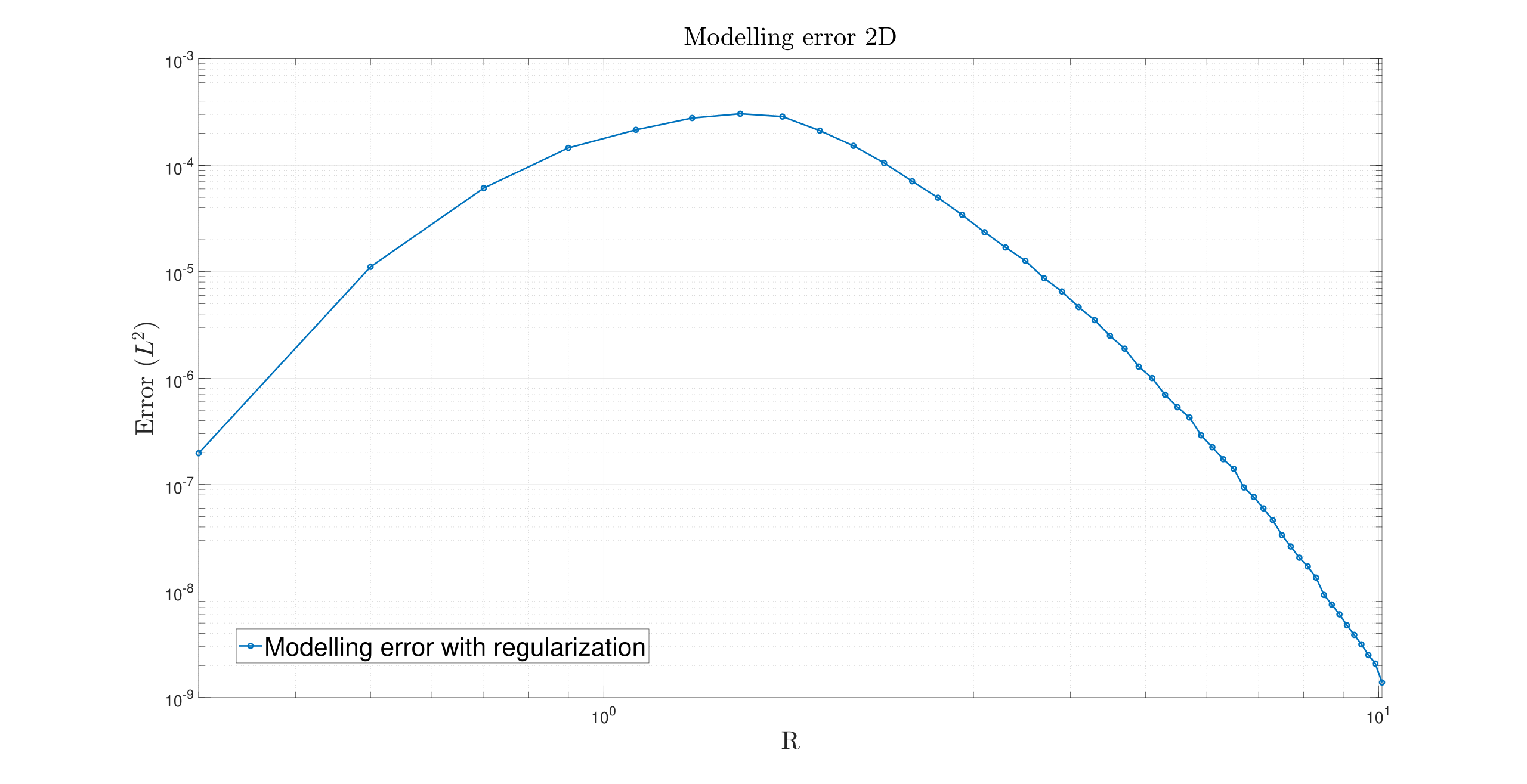}
    \caption{Relation between domain size $R$ and $L^2$-modelling-error for the 2D problem, exhibiting exponential decay}
    \label{fig:2dRvmodellingerr}
\end{figure}

\subsection{Quasi-Periodic Coefficient}

Finally, we examine the performance of the regularized method for a problem with a quasi-periodic coefficient. We consider a 2-dimensional case where the coefficient matrix $a$ is given by
\begin{equation}
    a(\boldsymbol{x})=0.25e^{\sin(2\sqrt{2}\pi x_1)+\sin(2\pi x_1)}e^{\sin(2\sqrt{2}\pi x_2)+\sin(2\pi x_2)}\mathbb{I}_2,
\end{equation}
where $\mathbb{I}_2$ is the $2\times 2$ identity matrix. The source term is defined as
\begin{equation}
    g(\boldsymbol{x})=(\sin(10\pi x_1)+\sin(10\pi x_2))(e^{\frac{1}{x_1^2+x_2^2+1}}-1).
\end{equation}
As a reference solution we use the solution computed with $R=15$ as our reference. Figure \ref{fig:quasiperiodicerror2d} demonstrates that the regularized method maintains its effectiveness in this quasi-periodic setting, exhibiting error decay behavior similar to the earlier cases.

\begin{figure}[H]
    \centering
    \includegraphics[width = 0.8\textwidth]{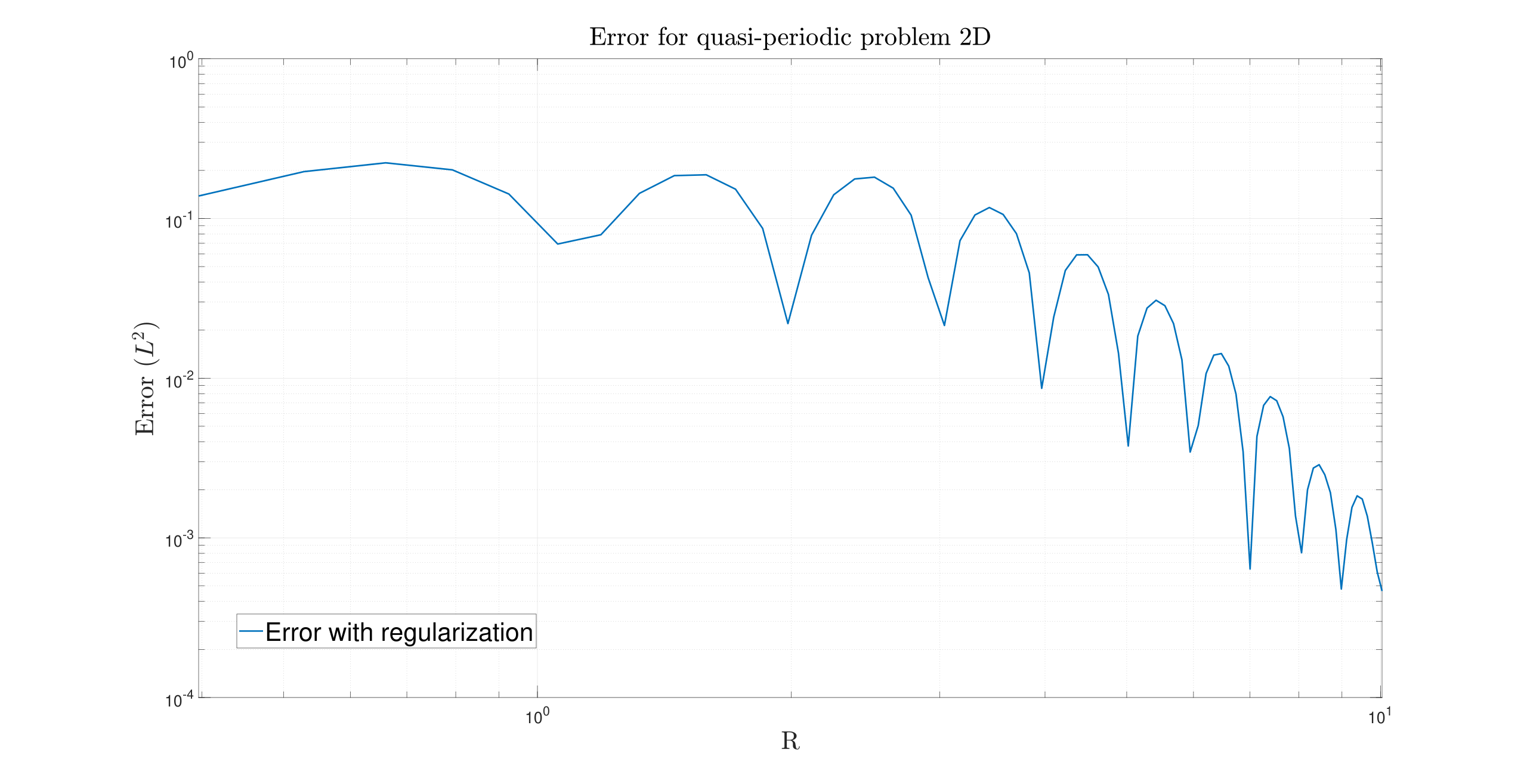}
    \caption{Relation between domain size $R$ and $L^2$-error for the quasi-periodic problem, showing the regularized method's robust performance}
    \label{fig:quasiperiodicerror2d}
\end{figure}

\section*{Acknowledgements}
This research project is funded by the grant number 2021-05493 from Vetenskapsrådet (VR).

\bibliographystyle{plain}
\bibliography{main}

\begin{thebibliography}{10}

\bibitem{AA_DA_EP_2019__357_6_545_0}
Assyr Abdulle, Doghonay Arjmand, and Edoardo Paganoni.
\newblock Exponential decay of the resonance error in numerical homogenization via parabolic and elliptic cell problems.
\newblock {\em Comptes Rendus. Math\'ematique}, 357(6):545--551, 2019.

\bibitem{abdulle2021parabolic}
Assyr Abdulle, Doghonay Arjmand, and Edoardo Paganoni.
\newblock A parabolic local problem with exponential decay of the resonance error for numerical homogenization.
\newblock {\em Mathematical Models and Methods in Applied Sciences}, 31(13):2733--2772, 2021.

\bibitem{abdulle2023elliptic}
Assyr Abdulle, Doghonay Arjmand, and Edoardo Paganoni.
\newblock An elliptic local problem with exponential decay of the resonance error for numerical homogenization.
\newblock {\em Multiscale Modeling \& Simulation}, 21(2):513--541, 2023.

\bibitem{aharoni2000introduction}
A.~Aharoni.
\newblock {\em Introduction to the Theory of Ferromagnetism}.
\newblock International series of monographs on physics. Oxford University Press, Oxford, 2000.

\bibitem{al2011computing}
Awad~H Al-Mohy and Nicholas~J Higham.
\newblock Computing the action of the matrix exponential, with an application to exponential integrators.
\newblock {\em SIAM journal on scientific computing}, 33(2):488--511, 2011.

\bibitem{arjmand2024hybrid}
Doghonay Arjmand and Victor~Martinez Calzada.
\newblock A hybrid boundary integral-pde approach for the approximation of the demagnetization potential in micromagnetics.
\newblock {\em arXiv preprint arXiv:2404.12284}, 2024.

\bibitem{arjmand2020modelling}
Doghonay Arjmand, Mikhail Poluektov, and Gunilla Kreiss.
\newblock Modelling long-range interactions in multiscale simulations of ferromagnetic materials.
\newblock {\em Advances in Computational Mathematics}, 46:1--31, 2020.

\bibitem{arjmand2016time}
Doghonay Arjmand and Olof Runborg.
\newblock A time dependent approach for removing the cell boundary error in elliptic homogenization problems.
\newblock {\em Journal of Computational Physics}, 314:206--227, 2016.

\bibitem{Armstrong2019QuantitativeSH}
Scott Armstrong, Tuomo Kuusi, and Jean-Christophe Mourrat.
\newblock {\em Quantitative Stochastic Homogenization and Large-Scale Regularity}.
\newblock Springer Cham, Cham, 2019.

\bibitem{Aronson_68}
D.~G. Aronson.
\newblock Non-negative solutions of linear parabolic equations.
\newblock {\em Annali della Scuola Normale Superiore di Pisa - Scienze Fisiche e Matematiche}, Ser. 3, 22(4):607--694, 1968.

\bibitem{berenger1994perfectly}
Jean-Pierre Berenger.
\newblock A perfectly matched layer for the absorption of electromagnetic waves.
\newblock {\em Journal of computational physics}, 114(2):185--200, 1994.

\bibitem{blanc2009improving}
Xavier Blanc and Claude Le~Bris.
\newblock Improving on computation of homogenized coefficients in the periodic and quasi-periodic settings.
\newblock 2009.

\bibitem{Carney_etal_2024}
Sean~P. Carney, Milica Dussinger, and Bj\"{o}rn Engquist.
\newblock On the nature of the boundary resonance error in numerical homogenization and its reduction.
\newblock {\em Multiscale Modeling \& Simulation}, 22(2):811--835, 2024.

\bibitem{cioranescu1999introduction}
Doina Cioranescu and Patrizia Donato.
\newblock {\em {An Introduction to Homogenization}}.
\newblock Oxford University Press, Oxford, 11 1999.

\bibitem{DAO2024113009}
Tuan~Anh Dao, Murtazo Nazarov, and Ignacio Tomas.
\newblock A structure preserving numerical method for the ideal compressible mhd system.
\newblock {\em Journal of Computational Physics}, 508:113009, 2024.

\bibitem{Duru_Kreiss_2012}
Kenneth Duru and Gunilla Kreiss.
\newblock A well-posed and discretely stable perfectly matched layer for elastic wave equations in second order formulation.
\newblock {\em Communications in Computational Physics}, 11(5):1643–1672, 2012.

\bibitem{Engquist_Majda_77}
Bjorn Engquist and Andrew Majda.
\newblock Absorbing boundary conditions for the numerical simulation of waves.
\newblock {\em Mathematics of Computation}, 31(139):629--651, 1977.

\bibitem{gilbarg1977elliptic}
David Gilbarg, Neil~S Trudinger, David Gilbarg, and NS~Trudinger.
\newblock {\em Elliptic partial differential equations of second order}, volume 224.
\newblock Springer, Heidelberg, 1977.

\bibitem{gloria2011reduction}
Antoine Gloria.
\newblock Reduction of the resonance error—part 1: Approximation of homogenized coefficients.
\newblock {\em Mathematical Models and Methods in Applied Sciences}, 21(08):1601--1630, 2011.

\bibitem{Hofman_Kim_2007}
Steve Hofmann and Kim Seick.
\newblock The green function estimates for strongly elliptic systems of second order.
\newblock {\em manuscripta math}, 124:139--172, 2007.

\bibitem{DUOANDIKOETXEA1992}
Duoandikoetxea J. and Zuazua E.
\newblock Moments, masses de dirac et décomposition de fonctions.
\newblock {\em Comptes rendus de l'Académie des sciences. Série 1, Mathématique}, 1992.

\bibitem{KANG20102643}
Kyungkeun Kang and Seick Kim.
\newblock Global pointwise estimates for green's matrix of second order elliptic systems.
\newblock {\em Journal of Differential Equations}, 249(11):2643--2662, 2010.

\bibitem{FeniCsBook}
Anders Logg, Kent-Andre Mardal, and Garth Wells.
\newblock {\em Automated Solution of Differential Equations by the Finite Element Method}.
\newblock Springer, 2012.

\bibitem{nabizadeh2021kelvin}
Mohammad~Sina Nabizadeh, Ravi Ramamoorthi, and Albert Chern.
\newblock Kelvin transformations for simulations on infinite domains.
\newblock {\em ACM Trans. Graph.}, 40(4):97--1, 2021.

\bibitem{PARDO2021219}
David Pardo, Paweł~J. Matuszyk, Vladimir Puzyrev, Carlos Torres-Verdín, Myung~Jin Nam, and Victor~M. Calo.
\newblock Chapter 7 - absorbing boundary conditions.
\newblock In David Pardo, Paweł~J. Matuszyk, Vladimir Puzyrev, Carlos Torres-Verdín, Myung~Jin Nam, and Victor~M. Calo, editors, {\em Modeling of Resistivity and Acoustic Borehole Logging Measurements Using Finite Element Methods}, pages 219--246. Elsevier, Amsterdam, 2021.

\bibitem{pavliotis2008multiscale}
G.A. Pavliotis and A.~Stuart.
\newblock {\em Multiscale Methods: Averaging and Homogenization}.
\newblock Texts in Applied Mathematics. Springer, New York, 2008.

\bibitem{F_O_Porper_1984}
F~O Porper and S~D Eidel'man.
\newblock Two-sided estimates of fundamental solutions of second-order parabolic equations, and some applications.
\newblock {\em Russian Mathematical Surveys}, 39(3):119, jun 1984.

\bibitem{taflove2005computational}
Allen Taflove, Susan~C Hagness, and Melinda Piket-May.
\newblock Computational electromagnetics: the finite-difference time-domain method.
\newblock {\em The Electrical Engineering Handbook}, 3(629-670):15, 2005.

\bibitem{VEDENYAPIN201119}
Victor Vedenyapin, Alexander Sinitsyn, and Eugene Dulov.
\newblock 3 - vlasov-maxwell and vlasov-einstein equations.
\newblock In Victor Vedenyapin, Alexander Sinitsyn, and Eugene Dulov, editors, {\em Kinetic Boltzmann, Vlasov and Related Equations}, pages 19--28. Elsevier, London, 2011.

\bibitem{weinan2007heterogeneous}
E~Weinan, Bjorn Engquist, Xiantao Li, Weiqing Ren, and Eric Vanden-Eijnden.
\newblock Heterogeneous multiscale methods: a review.
\newblock {\em Communications in computational physics}, 2(3):367--450, 2007.

\bibitem{yue2007local}
Xingye Yue and E~Weinan.
\newblock The local microscale problem in the multiscale modeling of strongly heterogeneous media: Effects of boundary conditions and cell size.
\newblock {\em Journal of Computational Physics}, 222(2):556--572, 2007.

\end{thebibliography}

\end{document}